\documentclass{article}

\usepackage{amsmath,amssymb,amsthm}

\usepackage[T1]{fontenc}
\usepackage{kpfonts, baskervald}
\numberwithin{equation}{section}
\newtheorem{thm}{Theorem}[section]
\newtheorem{cor}[thm]{Corollary}
\newtheorem{prop}[thm]{Proposition}
\newtheorem{lem}[thm]{Lemma}

\theoremstyle{definition}
\newtheorem{defn}[thm]{Definition}
\theoremstyle{remark}
\newtheorem{rmk}[thm]{Remark}
\newtheorem{exam}[thm]{Example}

\usepackage{microtype}

\usepackage[all]{xy}

\usepackage{todonotes}

\usepackage{hyperref}
\hypersetup{
  colorlinks   = true, 
  urlcolor     = blue, 
  linkcolor    = blue, 
  citecolor   = red 
}

\newcommand{\co}{\colon\thinspace}
\newcommand{\mb}[1]{\mathbb{#1}}
\newcommand{\mf}[1]{\mathfrak{#1}}

\newcommand{\too}{\xrightarrow}

\DeclareMathOperator{\Hom}{Hom}
\DeclareMathOperator{\Map}{Map}

\DeclareMathOperator{\End}{End}

\DeclareMathOperator{\Fun}{Fun}

\DeclareMathOperator*{\colim}{colim}
\DeclareMathOperator*{\hocolim}{hocolim}
\DeclareMathOperator*{\holim}{holim}

\title{An introduction to Bousfield localization}
\author{Tyler Lawson}

\begin{document}
\maketitle

\section{Introduction}

Bousfield localization encodes a wide variety of constructions in
homotopy theory, analogous to localization and completion in
algebra. Our goal in this chapter is to give an overview of Bousfield
localization, sketch how basic results in this area are proved, and
illustrate some applications of these techniques. Near the end we will
give more details about how localizations are constructed using the
small object argument. The underlying methods apply in many contexts,
and we have attempted to provide a variety of examples that exhibit
different behavior.

We will begin by discussing categorical localizations. Given a
collection of maps in a category, the corresponding localization of
that category is formed by making these maps invertible in a universal
way; this technique is often applied to discard irrelevant information
and focus on a particular type of phenomenon. In certain cases,
localization can be carried out internally to the category itself:
this happens when there is a sufficiently ample collection of objects
that already see these maps as isomorphisms. This leads naturally to
the study of reflective localizations.

Bousfield localization generalizes this by taking place in a category
where there are \emph{spaces} of functions, rather than sets, with
uniqueness only being true up to contractible choice. Bousfield
codified these properties, for spaces in
\cite{bousfield-spacelocalization} and for spectra in
\cite{bousfield-spectralocalization}. The definitions are
straightforward, but proving that localizations exist takes work, some
of it of a set-theoretic nature.

Our presentation is close in spirit to Bousfield's work, but the
reader should go to the books of Farjoun \cite{farjoun-localization}
and Hirschhorn \cite{hirschhorn} for more advanced information on this
material. We will focus, for the most part, on \emph{left} Bousfield
localization, since the techniques there are easier and is where most
of our applications lie. In \cite{barwick-bousfieldlocalization} right
Bousfield localization is discussed at more length.

\subsection{Historical background}

The story of localization techniques in algebraic topology probably
begins with Serre classes of abelian groups
\cite{serre-classes}. After choosing a class $\mathcal{C}$ of abelian
groups that is closed under subobjects, quotients, and extensions,
Serre showed that one could effectively ignore groups in $\mathcal{C}$
when studying the homology and homotopy of a simply-connected space
$X$. In particular, he proved mod-$\mathcal{C}$ versions of the
Hurewicz and Whitehead theorems, showed the equivalence between finite
generation of homology and homotopy groups, determined the rational
homotopy groups of spheres, and significantly reduced the technical
overhead in computing the torsion in homotopy groups by allowing one
to work with only one prime at a time. His techniques for computing
rational homotopy groups only require rational homology groups;
$p$-local homotopy groups only require $p$-local homology groups;
$p$-completed homotopy groups only require mod-$p$ homology groups.

These techniques received a significant technical upgrade in the late
1960's and early 1970's, starting with the work of Quillen on rational
homotopy theory \cite{quillen-rational} and work of Sullivan and
Bousfield--Kan on localization and completion of spaces
\cite{sullivan-mitnotes,sullivan-genetics,bousfield-kan}. Rather than
using Serre's algebraic techniques to break up the homotopy groups
$\pi_* X$ and homology groups $H_* X$ into localizations and
completions, their insight was that \emph{space-level} versions of
these constructions provided a more robust theory. For example, a
simply-connected space $X$ has an associated space $X_{\mb Q}$ whose
homotopy groups and (positive-degree) homology groups are, themselves,
rational homotopy and homology groups of $X$; similarly for Sullivan's
$p$-localization $X_{(p)}$ and $p$-completion $X^\wedge_p$. Without
this, each topological tool requires a proof that it is compatible
with Serre's mod-$\mathcal{C}$-theory, such as Serre's
mod-$\mathcal{C}$ Hurewicz and Whitehead theorems or mod-$\mathcal{C}$
cup products. Now these are simply consequences of the Hurewicz and
Whitehead theorems applied to $X_{\mb Q}$, and any subsequent
developments will automatically come along. Moreover, Sullivan
pioneered arithmetic fracture techniques that allowed $X$ to be
recovered from its rationalization $X_\mb Q$ and its $p$-adic
completions $X^\wedge_p$ via a homotopy pullback diagram:
\[
  \xymatrix{
    X \ar[r] \ar[d] &
    \prod_p X^\wedge_p \ar[d] \\
    X_\mb Q \ar[r]_-\alpha &
    (\prod_p X^\wedge_p)_{\mb Q}
  }
\]
This allows us to reinterpret homotopy theory. We are no longer using
rationalization and completion just to understand algebraic invariants
of $X$: instead, knowledge of $X$ is equivalent to knowledge of its
localizations, completions, and an ``arithmetic attaching map''
$\alpha$. This entirely changed both the way theorems are proved and
the way that we think about the subject. Later, work of Morava,
Ravenel, and others made extensive use of localization techniques
\cite{morava-noetherian, ravenel-localization}, which today gives an
explicit decomposition of the stable homotopy category into layers
determined by Quillen's relation to the structure theory of formal
group laws \cite{quillen-fgl}.

Many of the initial definitions of localization and completion were
constructive. One can build $X_{\mb Q}$ from $X$ by showing that one
can replace the basic cells $S^n$ in a CW-decomposition with
rationalized spheres $S^n_{\mb Q}$, or by showing that the
Eilenberg--Mac Lane spaces $K(A,n)$ in a Postnikov decomposition can
be replaced by rationalized versions $K(A \otimes \mb Q, n)$. One can
instead use Bousfield and Kan's more functorial, but also more
technical, construction as the homotopy limit of a cosimplicial
space. Quillen's work gives more, in the form of a model structure
whose weak equivalences are isomorphisms on rational homology
groups. In his work, the map $X \to X_{\mb Q}$ is a fibrant
replacement, and the essential uniqueness of fibrant replacements
means that $X_{\mb Q}$ has a form of universality. It is this
universal property that Bousfield localization makes into a
definition.

\subsection{Notation}

We will use $\mathcal{S}$ to denote an appropriately convenient
category of spaces (one can use simplicial sets, but with appropriate
modifications throughout) with internal function objects. We similarly
write $\mathcal{S}p$ for a category of spectra.

Throughout this paper we will often be working in categories
enriched in spaces: for any $X$ and $Y$ in $\mathcal{C}$ we will write
$\Map_\mathcal{C}(X,Y)$ for the mapping space, or just $\Map(X,Y)$ if
the ambient category is understood. Letting
$[X,Y] = \pi_0 \Map_{\mathcal{C}}(X,Y)$, we obtain an ordinary
category called the \emph{homotopy category} $h\mathcal{C}$. Two
objects in $\mathcal{C}$ are \emph{homotopy equivalent} if and only if
they become isomorphic in $h\mathcal{C}$.

For us, homotopy limits and colimits in the category of spaces are
given by the descriptions of Vogt or Bousfield--Kan
\cite{vogt-hocolim, bousfield-kan}. A homotopy limit or homotopy
colimit in $\mathcal{C}$ is characterized by having a natural weak
equivalence of spaces:
\begin{align*}
  \Map_{\mathcal{C}}(X, \holim_J Y_j)
  &\simeq \holim_J \Map_{\mathcal{C}}(X, Y_j)\\
  \Map_{\mathcal{C}}(\hocolim_I X_i, Y)
  &\simeq \holim_I \Map_{\mathcal{C}}(X_i, Y)
\end{align*}
In particular, since homotopy limit constructions on spaces preserve
objectwise weak equivalences of diagrams, homotopy limits and colimits
also preserve objectwise homotopy equivalences in $\mathcal{C}$.

Some set theory is unavoidable, but we will not spend a great deal of
time with it. For us, a \emph{collection} or \emph{family} may be a proper
class, rather than a set. Categories will be what are sometimes called
\emph{locally small} categories: the collection of objects may be
large, but there is a set of maps between any pair of objects.

\subsection{Acknowlegements}

The author would like to thank Clark Barwick and Thomas Nikolaus for
discussions related to localizations of spaces.

Th author was partially supported by NSF grant 1610408 and a grant
from the Simons Foundation.  The author would like to thank the Isaac
Newton Institute for Mathematical Sciences for support and hospitality
during the programme HHH when work on this paper was undertaken. This
work was supported by: EPSRC grant numbers EP/K032208/1 and
EP/R014604/1.

\section{Motivation from categorical localization}

In general, we recall that for an ordinary category $\mathcal{A}$ and
a class $\mathcal{W}$ of the maps called \emph{weak equivalences} (or
simply \emph{equivalences}), we can attempt to construct a categorical
localization $\mathcal{A} \to \mathcal{A}[\mathcal{W}^{-1}]$. This
localization is universal among functors $\mathcal{A} \to \mathcal{D}$
that send the maps in $\mathcal{W}$ to isomorphisms. The category
$\mathcal{A}[\mathcal{W}^{-1}]$ is unique up to isomorphism if it
exists.\footnote{For the record, this category also satisfies a
  2-categorical universal property: for any $\mathcal{D}$, the map of
  functor categories
  \[
    \Fun(\mathcal{A}[\mathcal{W}^{-1}], \mathcal{D}) \to
    \Fun(\mathcal{A},\mathcal{D})
  \]
  is fully faithful, and the image consists of those
  functors sending $\mathcal{W}$ to isomorphisms. If we replace
  ``image'' with ``essential image'' in this description, we recover a
  universal property characterizing $\mathcal{A} \to
  \mathcal{A}[\mathcal{W}^{-1}]$ up to equivalence of categories
  rather than up to isomorphism.}
  


\begin{exam}
  We will begin by remembering the case of the category $\mathcal{S}$
  of spaces, with $\mathcal{W}$ the class of weak homotopy
  equivalences. The projection $p\co X \times [0,1] \to X$ is always a
  weak equivalence with homotopy inverses $i_t$ given by
  $i_t(x) = (x,t)$. In the localization, we find that homotopic maps
  are equal: for a homotopy $H$ from $f$ to $g$, we have
  $f = H i_0 = H p^{-1} = H i_1 = g$. Therefore, localization factors
  through the homotopy category $h\mathcal{S}$.

  However, within the category of spaces we have a collection with
  special properties: the subcategory $\mathcal{S}^{CW}$ of
  CW-complexes. For any CW-complex $K$, weak equivalences $X \to Y$
  induce bijections $[K,X] \to [K,Y]$---this can be proved, for
  example, inductively on the cells of $K$---and any space $X$ has a
  CW-complex $K$ with a weak homotopy equivalence $K \to X$. These two
  properties show, respectively, that the composite
  \[
    h\mathcal{S}^{CW} \to h\mathcal{S} \to
    \mathcal{S}[\mathcal{W}^{-1}]
  \]
  is fully faithful and essentially surjective. Within the homotopy
  category $h\mathcal{S}$ we have found a large enough library of
  special objects, and localization can be done by forcibly moving
  objects into this subcategory.\footnote{Technically speaking, we
    often use a result like this to actually show that
    $\mathcal{S}[\mathcal{W}^{-1}]$ exists.}
\end{exam}

\begin{exam}
  A similar example to the above occurs in the category
  $\mathcal{K}_R$ of nonnegatively graded cochain complexes of modules
  over a commutative ring $R$, with $\mathcal{W}$ the class of
  quasi-isomorphisms. Within $\mathcal{K}_R$ there is a subcategory
  $\mathcal{K}_R^{Inj}$ of complexes of injective modules. Fundamental
  results of homological algebra show that for a quasi-isomorphism
  $A \to B$ and a complex $Q$ of injectives, there is a bijection
  $[B,Q] \to [A,Q]$ of chain homotopy classes of maps, and that any
  complex $A$ has a quasi-isomorphism $A \to Q$ to a complex of
  injectives. This similarly shows that the composite functor 
  \[
    h\mathcal{K}_R^{Inj} \to h\mathcal{K}_R \to \mathcal{K}_R[\mathcal{W}^{-1}]
  \]
  is an equivalence of categories.
\end{exam}

These examples are at the foundation of Quillen's theory of model
categories, and we will return to examples like them when we discuss
localization of model categories.

\section{Local objects in categories}
\label{sec:local-objects}

In this section we will fix an ordinary category $\mathcal{A}$.

\begin{defn}
  Let $S$ be a class of morphisms in $\mathcal{A}$.  An object
  $Y \in \mathcal{A}$ is \emph{$S$-local} if, for all $f\co A \to B$
  in $S$, the map
  \[
    \Hom_{\mathcal{A}}(B, Y) \too{f^*}
    \Hom_{\mathcal{A}}(A, Y)
  \]
  is a bijection. We write $L^S(\mathcal{A})$ for the full subcategory
  of $S$-local objects.
  
  If $S = \{f\co A \to B\}$ consists of just one map, we simply refer
  to this property as being \emph{$f$-local} and write
  $L^f(\mathcal{A})$ for the category of $f$-local objects.
\end{defn}

\begin{rmk}
  If $S = \{f_\alpha\co A_\alpha \to B_\alpha\}$ is a set and
  $\mathcal{A}$ has coproducts indexed by $S$, then by defining
  $f = \coprod_\alpha f_\alpha\co \coprod A_\alpha \to \coprod
  B_\alpha$ we find that $S$-local objects are equivalent to $f$-local
  objects.
\end{rmk}

A special case of localization is when our maps in $S$ are maps to a
terminal object.
\begin{defn}
  Suppose $S$ is a class of maps $\{W_\alpha \to *\}$, where $*$ is a
  terminal object. In this case, we refer to such a localization as a
  \emph{nullification} of the objects $W_\alpha$.
\end{defn}

\begin{rmk}
  Nullification often takes place when $\mathcal{A}$ is pointed. If
  $S$ is a set, $\mathcal{A}$ is pointed, and $\mathcal{A}$ has
  coproducts, then any coproduct of copies of $*$ is again $*$ and we
  can again replace nullification of a set of objects with
  nullification of an individual object.
\end{rmk}

\begin{defn}
  A map $A \to B$ in $\mathcal{A}$ is an \emph{$S$-equivalence} if,
  for all $S$-local objects $Y$, the map
  \[
    \Hom_{\mathcal{A}}(B,Y) \to \Hom_{\mathcal{A}}(A,Y)
  \]
  is a bijection.
\end{defn}
The class of $S$-equivalences contains $S$ by definition.

\begin{defn}
  A map $X \to Y$ is an \emph{$S$-localization} if it is an
  $S$-equivalence and $Y$ is $S$-local, and under these conditions we
  say that $X$ \emph{has an $S$-localization}. If all objects in
  $\mathcal{A}$ have $S$-localizations, we say that
  \emph{$\mathcal{A}$ has $S$-localizations.}
\end{defn}

\begin{prop}
  \label{prop:uniquelocalization}
  Any two $S$-localizations $f_1\co X \to Y_1$ and $f_2\co X \to Y_2$
  are isomorphic under $X$ in $\mathcal{A}$.
\end{prop}

\begin{proof}
  Because $Y_i$ are $S$-local, $\Hom(B,Y_i) \to \Hom(A,Y_i)$ is always
  an isomorphism for any $S$-equivalence $A \to B$. Applying this to
  the $S$-equivalences $X \to Y_j$, we get isomorphisms
  $\Hom(Y_j, Y_i) \to \Hom(X,Y_i)$ in $\mathcal{A}$: any map
  $X \to Y_i$ has a unique extension to a map $Y_j \to Y_i$. Existence
  allows us to find maps $Y_1 \to Y_2$ and $Y_2 \to Y_1$ under $X$,
  and uniqueness allows us to conclude that these two maps are inverse
  to each other in $\mathcal{A}$.

  More concisely, $Y_1$ and $Y_2$ are both initial objects in the
  comma category of $S$-local objects under $X$ in $\mathcal{A}$, and
  this universal property forces them to be isomorphic.
\end{proof}

As a result, it is reasonable to call such an object \emph{the}
$S$-localization of $X$ and write it as $L^S X$ (or simply $LX$ if $S$
is understood). More generally than this, if $X \to LX$ and
$X' \to LX'$ are $S$-localization maps, any map $X \to X'$ in
$\mathcal{A}$ extends uniquely to a commutative square. This is
encoded by the following result.
\begin{prop}
  \label{prop:chooselocalizations}
  Let $Loc^S(\mathcal{A})$ be the category of \emph{localization
    morphisms}, whose objects are $S$-localization maps
  $X \to LX$ in $\mathcal{A}$ and whose morphisms are commuting
  squares. Then the forgetful functor
  \[
    Loc^S(\mathcal{A}) \to \mathcal{A},
  \]
  sending $(X \to LX)$ to $X$, is fully faithful. The image consists
  of those objects $X$ that have $S$-localizations.
\end{prop}

\begin{prop}
  The collection of $S$-local objects is closed under limits,
  and the collection of $S$-equivalences is closed under 
  colimits.
\end{prop}

\begin{proof}
  If $f\co A \to B$ is in $S$ and $\{Y_j\}$ is a diagram of $S$-local
  objects, then
  \[
    \Hom(B,Y_j) \to \Hom(A,Y_j)
  \]
  is a diagram of isomorphisms, and taking limits
  we find that we have an isomorphism
  \[
    \Hom(B,\lim_J Y_j) \to \Hom(A,\lim_J Y_j).
  \]
  Since $A \to B$ was an arbitrary map in $S$, this shows that
  $\lim_J Y_j$ is $S$-local.
  
  Similarly, if $\{A_i \to B_i\}$ is a diagram of $S$-equivalences and
  $Y$ is $S$-local, then
  \[
    \Hom(B_i, Y) \to \Hom(A_i,Y)
  \]
  is a diagram of isomorphisms, and taking limits we find that
  \[
    \Hom(\colim_I B_i, Y) \to \Hom(\colim_I A_i, Y)
  \]
  is also an isomorphism. Since $Y$ was an arbitrary local object, this
  shows that the map $\colim_I A_i \to \colim_I B_i$ is an
  $S$-equivalence.
\end{proof}

\begin{exam}
  Consider the map $f\co \mb N \to \mb Z$ in the category of
  monoids. A monoid $M$ is $f$-local if and only if any monoid
  homomorphism $\mb N \to M$ automatically extends to a homomorphism
  $\mb Z \to M$, which is the same as asking that every element in $M$
  has an inverse. Therefore, $f$-local monoids are precisely
  \emph{groups}. The natural transformation $M \to M^{gp}$, from a
  monoid to its group completion, is an $f$-localization.
\end{exam}

\begin{exam}
  Consider the map $f\co F_2 \to \mb Z^2$, from a free group on two
  generators $x$ and $y$ to its abelianization. A group $G$ is $f$-local
  if and only if \emph{every} homomorphism $F_2 \to G$, equivalent to
  choosing a pair of elements $x$ and $y$ of $G$, can be factored
  through $\mb Z^2$, which happens exactly when the commutator $[x,y]$
  is sent to the trivial element. Therefore, $f$-local groups are
  precisely \emph{abelian} groups. The natural transformation $G \to
  G_{ab}$, from a group to its abelianization, is an
  $f$-localization.
\end{exam}

These two localizations are left adjoints to the inclusion of a
subcategory, and this phenomenon is completely general.
\begin{prop}
  Let $S$ be a class of morphisms in $\mathcal{A}$, and suppose that
  $\mathcal{A}$ has $S$-localizations. Then the inclusion
  $L^S \mathcal{A} \to \mathcal{A}$ is part of an adjoint pair
  \[
    \mathcal{A} \stackrel{L}{\rightleftarrows} L^S \mathcal{A}.
  \]
  As a result, $L$ is a reflective localization onto the subcategory
  $L^S \mathcal{A}$.
\end{prop}

\begin{proof}
  In this situation, the functor $Loc^S(\mathcal{A}) \to \mathcal{A}$
  is fully faithful and surjective on objects. Therefore, it is an
  equivalence of categories and we can choose an inverse,\footnote{If
    the category $\mathcal{A}$ is large then we need to be a little
    bit more honest here, and worry about whether a fully faithful and
    essentially surjective map between large categories has an inverse
    equivalence. This depends on our model for set theory: it is
    asking for us to make a distinguished choice of objects for our
    inverse functor, which may require an axiom of choice for proper
    classes. It is an awkward situation, because choosing these
    inverses isn't categorically interesting unless we can't do it.}
  functorially sending $X$ to a pair $(X \to LX)$ in
  $Loc^S(\mathcal{A})$. The composite functor sending $X$ to $LX$ is
  the desired left adjoint.
\end{proof}

\begin{rmk}
  Embedding the category $\mathcal{A}$ as a full subcategory of a
  larger category can change localization drastically. Consider a set
  $S$ of maps in $\mathcal{A} \subset \mathcal{B}$. Then the $S$-local
  objects of $\mathcal{A}$ are simply the $S$-local objects of
  $\mathcal{B}$ that happen to be in $\mathcal{A}$, but because there
  may be more local objects in $\mathcal{B}$ there may be fewer
  $S$-equivalences in $\mathcal{B}$ than in
  $\mathcal{A}$. Localization in $\mathcal{B}$ may not preserve
  objects of $\mathcal{A}$; a localization map in $\mathcal{A}$ might
  not be an equivalence in $\mathcal{B}$; there might, in general, be
  no comparison map between the two localizations.

  For example, consider the set $S$ of multiplication-by-$p$ maps
  $\mb Z \to \mb Z$ (as $p$ ranges over primes) in the category of
  finitely generated abelian groups, considered as a full subcategory
  of all abelian groups. An abelian group is $S$-local if and only if
  it is a rational vector space, and the only finitely generated group
  of this form is trivial. A map $A \to B$ of finitely generated
  abelian groups is an $S$-equivalence in the larger category of all
  abelian groups if and only if it induces an isomorphism
  $A \otimes \mb Q \to B \otimes \mb Q$, whereas it is \emph{always}
  an equivalence within the smaller category of finitely generated
  abelian groups. Within all abelian groups, $S$-localization is
  rationalization, whereas within finitely generated abelian groups,
  $S$-localization takes all groups to zero.
\end{rmk}

\section{Localization using mapping spaces}
\label{sec:mappingspaces}

We now consider the case where $\mathcal{C}$ is a category enriched in
spaces. The previous definitions and results apply perfectly well to
the homotopy category $h\mathcal{C}$. The following illustrates that
the homotopy category may be an inappropriate place to carry out such
localizations.

\begin{exam}
  Let us start with the homotopy category of spaces $h\mathcal{S}$,
  and fix an $n \geq 0$. Suppose that we want to invert the inclusion
  $S^n \to D^{n+1}$. We fairly readily find that any space $X$ has a
  map $X \to X'$ such that $[D^{n+1},X'] \to [S^n,X']$ is an
  isomorphism: construct $X'$ by attaching $(n+1)$-dimensional cells
  to $X$ until the $n$'th homotopy group $\pi_n(X',x) = 0$ is trivial
  at any basepoint.

  However, this construction lacks \emph{universality}. If $Y$ is any
  other space whose $n$'th homotopy groups are trivial, then any map
  $X \to Y$ can be extended to a map $X' \to Y$ because the attaching
  maps for the cells of $X'$ are trivial. However, this extension is
  \emph{not unique} up to homotopy: any two extensions
  $D^{n+1} \to X' \to Y$ of a cell $S^n \to X \to Y$ glue together to
  an obstruction class in $[S^{n+1},Y]$. As a result, if we construct
  two spaces $X'$ and $X''$ as attempted localizations of $X$, we can
  find maps $X' \to X''$ and $X'' \to X'$ but cannot establish
  that they are mutually inverse in the homotopy category.

  In short, in order for $Y$ to have \emph{uniqueness} for
  filling maps from $n$-spheres, we have to have \emph{existence} for
  filling maps from $(n+1)$-spheres. Thus, to make this localization
  work canonically we would need to enlarge our class $S$ to contain
  $S^{n+1} \to D^{n+2}$. The same argument then repeats, showing that
  a canonical localization for $S$ requires that $S$ also contain
  $S^m \to D^{m+1}$ for $m \geq n$.
\end{exam}

The example in the previous section leads to the following
principle. In our definitions, we must replace isomorphism on the path
components of mapping spaces with homotopy equivalence.

\begin{defn}
  Let $S$ be a class of morphisms in the category $\mathcal{C}$.  An
  object $Y \in \mathcal{C}$ is \emph{$S$-local} if, for all
  $f\co A \to B$ in $S$, the map
  \[
    \Map_{\mathcal{C}}(B, Y) \too{f^*}
    \Map_{\mathcal{C}}(A, Y)
  \]
  is a weak equivalence.\footnote{Note that the homotopy class of the
    map $\Map_{\mathcal{C}}(B,Y) \to \Map_{\mathcal{C}}(A,Y)$ only
    depends on the image of $f\co A \to B$ in the homotopy category
    $h\mathcal{C}$, and so we may simply view $S$ as a collection of
    representatives for a class of maps $\bar S$ in $h\mathcal{C}$.}
  We write $L^S(\mathcal{C})$ for the full subcategory of $S$-local
  objects.
  
  If $S = \{f\co A \to B\}$ consists of just one map, we simply refer
  to this property as being \emph{$f$-local} and write
  $L^f(\mathcal{C})$ for the category of $f$-local objects.
\end{defn}

\begin{defn}
  A map $A \to B$ in $\mathcal{C}$ is an \emph{$S$-equivalence} if,
  for all $S$-local objects $Y$, the map
  \[
    \Map_{\mathcal{C}}(B,Y) \to \Map_{\mathcal{C}}(A,Y)
  \]
  is a weak equivalence.
\end{defn}

\begin{defn}
  A map $X \to Y$ is an \emph{$S$-localization} if it is an
  $S$-equivalence and $Y$ is $S$-local, and under these conditions we
  say that $X$ \emph{has an $S$-localization}. If all objects in
  $\mathcal{C}$ have $S$-localizations, we say that
  \emph{$\mathcal{C}$ has $S$-localizations.}
\end{defn}

By applying $\pi_0$ to mapping spaces, we find that some of this
passes to the homotopy category.

\begin{prop}
  Let $\bar S$ be the image of $S$ in the homotopy category
  $h\mathcal{C}$. If $Y$ is $S$-local in $\mathcal{C}$, then its image
  in the homotopy category $h\mathcal{C}$ is $\bar S$-local. 
\end{prop}

\begin{rmk}
  An $S$-equivalence in $\mathcal{C}$ does not necessarily becomes an
  $\bar S$-equivalence in $h\mathcal{C}$ because there is potentially
  a larger supply of $\bar S$-local objects.
\end{rmk}

\begin{prop}
  Any two $S$-localizations $f_1\co X \to Y_1$ and $f_2\co X \to Y_2$
  become isomorphic under $X$ in the homotopy category
  $h\mathcal{C}$.
\end{prop}

\begin{proof}
  This proceeds exactly as in the proof of
  Proposition~\ref{prop:uniquelocalization}. Applying
  $\Map_\mathcal{C}(-,Y_i)$ to the $S$-equivalence $X \to Y_j$, we
  find that the maps $X \to Y_i$ extend to maps $Y_j \to Y_i$ which
  are unique up to homotopy. By first taking $i \neq j$ we construct
  maps between the $Y_i$ whose restrictions to $X$ are homotopic to
  the originals, and taking $i = j$ shows that the double composites
  are homotopic under $X$.
\end{proof}

\begin{rmk}
  At this point it would be very useful to show that, if they exist,
  localizations can be made functorial in the spirit of
  Proposition~\ref{prop:chooselocalizations}. There is typically no
  easy way to produce a functorial localization because many choices
  are made up to homotopy equivalence, and this leads to coherence
  issues: for example, if we have a diagram
  \[
    \xymatrix{
      X \ar[r] \ar[d] & X' \ar[d] \\
      LX \ar@{.>}[r] & LX'
    }
  \]
  where the vertical maps are $S$-localization, then we can construct
  at best the dotted map together with a \emph{homotopy} between the
  two double composites. Larger diagrams do get more extensive
  families of homotopies, but these take work to describe. This is a
  \emph{rectification problem} and in general it is not solvable
  without asking for more structure on $\mathcal{C}$. The small object
  argument, which we will discuss in \S\ref{sec:small-object}, can
  often be done carefully enough to give some form of functorial
  construction of the localization.
\end{rmk}

\begin{prop}
  \label{prop:saturation-omnibus}
  The following properties hold for a class $S$ of morphisms in
  $\mathcal{C}$.
  \begin{enumerate}
  \item The collection of $S$-local objects is closed under
    equivalence in the homotopy category.
  \item The collection of $S$-equivalences is closed under equivalence
    in the homotopy category.
  \item The collection of $S$-local objects is closed under homotopy
    limits.
  \item The collection of $S$-equivalences is closed under
    homotopy colimits.
  \item The homotopy pushout of an $S$-equivalence is
    an $S$-equivalence.
  \item The $S$-equivalences satisfy the two-out-of-three axiom: given
    maps $A \too{f} B \too{g} C$, if any two of $f$, $g$, and $gf$ are
    $S$-equivalences then so is the third.
  \end{enumerate}
\end{prop}

\begin{proof}
  If $X \to Y$ becomes an isomorphism in the homotopy category, then
  one can choose an inverse map and homotopies between the double
  composites. Composing with these makes
  $\Map_{\mathcal{C}}(-,X) \to \Map_{\mathcal{C}}(-,Y)$ a homotopy
  equivalence of functors on $\mathcal{C}$, and so $X$ is $S$-local if
  and only if $Y$ is.

  Similarly, if two maps $f\co A \to B$ and $f'\co A' \to B'$ become
  isomorphic in the homotopy category, there exist homotopy
  equivalences $A' \to A$ and $B \to B'$ such that the composite
  $A' \to A \to B \to B'$ is homotopic to $f'$, and applying
  $\Map_{\mathcal{C}}(-,Y)$ we obtain the desired result.
  
  If $f\co A \to B$ is in $S$ and $\{Y_j\}$ is a diagram of $S$-local
  objects, then
  \[
    \Map_{\mathcal{C}}(B,Y_j) \to \Map_{\mathcal{C}}(A,Y_j)
  \]
  is a diagram of weak equivalences of spaces, and taking homotopy
  limits we find that we have an equivalence
  \[
    \Map_{\mathcal{C}}(B,\holim_J Y_j) \to \Map_{\mathcal{C}}(A,\holim_J Y_j).
  \]
  Since $A \to B$ was an arbitrary map in $S$, this shows that
  $\holim_J Y_j$ is $S$-local.
  
  Similarly, if $\{A_i \to B_i\}$ is a diagram of $S$-equivalences and
  $Y$ is $S$-local, then
  \[
    \Map_{\mathcal{C}}(B_i, Y) \to \Map_{\mathcal{C}}(A_i,Y)
  \]
  is a diagram of weak equivalences of spaces, and so
  \[
    \Map_{\mathcal{C}}(\hocolim_I B_i, Y) \to \Map_{\mathcal{C}}(\hocolim_I A_i, Y)
  \]
  is also a weak equivalence. Since $Y$ was an arbitrary $S$-local
  object, this shows that the map $\hocolim_I A_i \to \hocolim_I B_i$
  is an $S$-equivalence.

  Suppose that we have a homotopy pushout diagram
  \[
    \xymatrix{
      A \ar[r]^f \ar[d] & B \ar[d] \\
      A' \ar[r]_{f'} & B'
    }
  \]
  where $f\co A \to B$ is an $S$-equivalence. Given any $S$-local
  object $Y$, we get a homotopy pullback diagram
  \[
    \xymatrix{
      \Map_{\mathcal{C}}(A,Y) & \Map_{\mathcal{C}}(B,Y) \ar[l] \\
      \Map_{\mathcal{C}}(A',Y)\ar[u] & \Map_{\mathcal{C}}(B',Y) \ar[l] \ar[u].
    }
  \]
  The top arrow is an equivalence by the assumption that $f$ is an
  $S$-equivalence, and hence the bottom arrow is an equivalence. Since
  $Y$ was an arbitrary $S$-local object, we find that $f'$ is an
  $S$-equivalence.

  The 2-out-of-3 property is obtained by first applying
  $\Map_{\mathcal{C}}(-,Y)$ to the diagram $A \to B \to C$ and then
  using the 2-out-of-3 axiom for weak equivalences.
\end{proof}

If we expand a class $S$ to a larger class $T$ of equivalences, our work
so far gives us an automatic relation between $S$-localization and
$T$-localization.
\begin{prop}
  Suppose that $S$ and $T$ are classes of morphisms such that every
  map in $S$ is a $T$-equivalence. Then the following properties
  hold.
  \begin{enumerate}
  \item Every $T$-local object is also $S$-local.
  \item Every $S$-equivalence is also a $T$-equivalence.
  \item Suppose $X \to L_S X$ is an $S$-localization and $X \to L_T X$
    is a $T$-localization. Then there exists an essentially unique
    factorization $X \to L_S X \to L_T X$, and the map $L_S X \to L_T
    X$ is a $T$-localization.
  \end{enumerate}
\end{prop}

\begin{proof}
  \begin{enumerate}
  \item By assumption, every map $f\co A \to B$ in $S$ is a
    $T$-equivalence, and so for any $T$-local object $Y$ we get an
    equivalence $\Map_{\mathcal{C}}(B,Y) \to
    \Map_{\mathcal{C}}(A,Y)$. Thus by definition $Y$ is $S$-local.
  \item If $f\co A \to B$ is an $S$-equivalence, and $Y$ is any
    $T$-local object, then by the previous point $Y$ is also
    $S$-local, and so we get an equivalence
    $\Map_{\mathcal{C}}(B,Y) \to \Map_{\mathcal{C}}(A,Y)$. Since $Y$
    was an arbitrary $T$-local object, $f$ is therefore a
    $T$-equivalence.
  \item Since $X \to L_S X$ is an $S$-equivalence, the previous point
    shows that it is a $T$-equivalence and so we have an equivalence
    \[
      \Map_{\mathcal{C}}(L_S X, L_T X) \to \Map_{\mathcal{C}}(X, L_T
      X).
    \]
    As a result, the chosen map $X \to L_T X$ has a contractible space
    of homotopy commuting factorizations $X \to L_S X \to L_T X$. As
    the maps $X \to L_S X$ and $X \to L_T X$ are both
    $T$-equivalences, the 2-out-of-3 property implies that $L_S X \to
    L_T X$ is also a $T$-equivalence whose target is $T$-local. By
    definition, this makes $L_T X$ into a $T$-localization of $L_S X$.\qedhere
  \end{enumerate}
\end{proof}

\section{Lifting criteria for localizations}
\label{sec:lifting}

In this section we will observe that, if $\mathcal{C}$ has homotopy
pushouts, we can characterize local objects in terms of a lifting
criterion. To do so, we will need to establish a few preliminaries.
Fix a collection $S$ of maps in $\mathcal{C}$.

\begin{prop}
  Suppose that $f\co A \to B$ is an $S$-equivalence, and that
  $\mathcal{C}$ has homotopy pushouts. Then the map
  \[
    \hocolim (B \leftarrow A \to B) \to B
  \]
  is an $S$-equivalence.
\end{prop}

\begin{proof}
  The map in question is equivalent to the map of homotopy pushouts
  induced by the diagram
  \[
    \xymatrix{
      B \ar@{=}[d] & A \ar[l]_f \ar[r]^f \ar[d]_f & B \ar@{=}[d] \\
      B & B \ar[l] \ar[r] & B.
    }
  \]
  However, the vertical maps are $S$-equivalences, and so by
  Proposition~\ref{prop:saturation-omnibus} the map $\hocolim(B
  \leftarrow A \to B) \to B$ is an $S$-equivalence.
\end{proof}

The lifting criterion we are about to describe rests on the following
useful characterization of connectivity of a map.

\begin{lem}
  \label{lem:inductive-extension}
  Suppose that $f\co X \to Y$ is a map of spaces and $N \geq 0$. Then
  $f$ is $N$-connected if and only if the following two criteria are
  satisfied:
  \begin{enumerate}
  \item the map $\pi_0(X) \to \pi_0(Y)$ is surjective, and
  \item the diagonal map $X \to\holim(X \to Y \leftarrow X)$ is
    $(N-1)$-connected.
  \end{enumerate}
\end{lem}

\begin{proof}
  The map $f$ is $N$-connected if and only if it is surjective on
  $\pi_0$ and, for all basepoints $x \in X$, the homotopy fiber $Ff$
  over $f(x)$ is $(N-1)$-connective.

  However, $Ff$ is equivalent to the homotopy fiber of
  $\holim(X \to Y \leftarrow X) \to X$ over $x$, and so this second
  condition is equivalent to $\holim(X \to Y \leftarrow X) \to X$
  being $N$-connected. The composite
  $X \to\holim(X \to Y \leftarrow X) \to X$ is the identity, and the
  map $\holim(X \to Y \leftarrow X) \to X$ is $N$-connected if and
  only if the map $X \to\holim(X \to Y \leftarrow X)$ is
  $(N-1)$-connected.
\end{proof}

\begin{cor}
  Suppose that $\mathcal{C}$ has homotopy pushouts and that we have a
  map $f_0\co A_0 \to B$ in $\mathcal{C}$. Inductively define the
  $n$-fold double mapping cylinder $f_n$ as the map
  \[
    A_n = \hocolim(B \leftarrow A_{n-1} \to B) \to B.
  \]
  Then an object $Y$ is $f_0$-local if and only if the maps
  \[
    \Hom_{h\mathcal{C}}(B,Y) \to \Hom_{h\mathcal{C}}(A_n,Y)
  \]
  are surjective; equivalently, for any map $A_n \to Y$ there is a map
  $B \to Y$ such that the diagram
  \[
    \xymatrix{
      A_n \ar[r] \ar[d]_{f_n} & Y\\
      B \ar@{.>}[ur]
    }
  \]
  is homotopy commutative.
\end{cor}

\begin{proof}
  We note that the definition of $A_n$ gives an identification
  \[
    \Map_{\mathcal{C}}(A_n,Y) \simeq \holim\left[\Map_{\mathcal{C}}(B,Y)
    \to \Map_{\mathcal{C}}(A_{n-1},Y) \leftarrow
    \Map_{\mathcal{C}}(B,Y)\right].
  \]
  Inductive application of Lemma~\ref{lem:inductive-extension} shows
  that the map
  \[
    \Map_{\mathcal{C}}(B,Y) \to \Map_{\mathcal{C}}(A_0,Y)
  \]
  is $N$-connected if and only if the maps
  \[
    \Hom_{h\mathcal{C}}(B,Y) \to \Hom_{h\mathcal{C}}(A_n,Y)
  \]
  are surjective for $0 \leq n \leq N$. Letting $N$ grow arbitrarily
  large, we find that $Y$ is $f_0$-local if and only of the maps
  \[
    \Hom_{h\mathcal{C}}(B,Y) \to \Hom_{h\mathcal{C}}(A_n,Y)
  \]
  are surjective for all $n \geq 0$.
\end{proof}

\begin{exam}
  Suppose that $\mathcal{C}$ has homotopy pushouts and that $f\co W
  \to *$ is a map to a homotopy terminal object of $\mathcal{C}$. Then
  the iterated double mapping cylinders are the maps $\Sigma^t
  W \to *$, and an object of $\mathcal{C}$ is $f$-local if and only if
  every map $\Sigma^t W \to Y$ factors, up to homotopy, through $*$.
\end{exam}

\begin{exam}
  In the category of spaces $\mathcal{S}$, the iterated double mapping
  cylinders $f_n$ of a cofibration $f_0\co A \to B$ have a more
  familiar description as the \emph{pushout-product} maps
  \[
    (S^{n-1} \times B) \mathop\cup_{S^{n-1} \times A} (D^n \times A) \to D^n
    \times B \to B.
  \]
\end{exam}

\section{The small object argument}
\label{sec:small-object}

We now sketch how, when we have some form of colimits in our category,
Bousfield localizations can often be constructed using the small object
argument.

From the previous section we know that we can replace the mapping
space criterion for local objects with a lifting criterion when
$\mathcal{C}$ has homotopy colimits, as follows. Given a map
$f_0\co A_0 \to B$, we construct iterated double mapping cylinders
$f_n\co A_n \to B$, and we find that an object is $Y$ is $f_0$-local
if and only if every map $g\co A_n \to Y$ can be extended to a map
$\tilde g\co B \to Y$ up to homotopy. More generally we can enlarge a
collection of maps $S$ to a collection $T$ closed under double mapping
cylinders, and ask whether $Y$ satisfies an extension property with
respect to $T$.

This leads to an inductive method.
\begin{enumerate}
\item Start with $Y_0= Y$.
\item Given $Y_\alpha$, either $Y_\alpha$ is local (in which case we
  are done) or there exists some set of maps $A_i \to B_i$ in
  $T$ and maps $g_i\co A_i \to Y_\alpha$ which do not extend to
  $B_i$ up to homotopy. Form the homotopy pushout of the diagram
  \[
    \coprod_i B_i \leftarrow \coprod_i A_i \rightarrow Y_\alpha
  \]
  and call it $Y_{\alpha+1}$. The map $Y_\alpha \to Y_{\alpha+1}$ is
  an $S$-equivalence because it is a homotopy pushout along an
  $S$-equivalence, and all the extension problems that
  $Y_\alpha$ had now have solutions in $Y_{\alpha+1}$.
\item Once we have constructed $Y_0, Y_1, Y_2, \dots$, define
  $Y_\omega = \hocolim Y_n$. More generally, once we have constructed
  $Y_\alpha$ for all ordinals $\alpha$ less than some limit ordinal
  $\beta$, we define $Y_\beta = \hocolim Y_\alpha$. The map
  $Y \to Y_\beta$ is a homotopy colimit of $S$-equivalences and hence
  an $S$-equivalence.
\end{enumerate}

The critical thing that we need is that \emph{this procedure can be
  stopped at some point}, and for this we typically need to know that
there will be some ordinal $\beta$ which is so big that any map
$A_i \to Y_\beta$ automatically factors, up to homotopy, through some
object $Y_\alpha$ with $\alpha < \beta$. This is a \emph{compactness}
property of the objects $A_i$, and this argument is called the
\emph{small object argument}. If we work on the point-set level this
can be addressed using Smith's theory of combinatorial model
categories; if we work on the homotopical level this can be addressed
using Lurie's theory of presentable $\infty$-categories. We will
discuss these approaches in \S\ref{sec:model-categories} and
\S\ref{sec:infty-categories}.

Another important aspect of the small object argument is that it can
prove additional properties about localization maps. If $S$ is a
collection of maps all satisfying some property $P$ of maps in the
homotopy category, and property $P$ is preserved under homotopy pushouts
and transfinite homotopy colimits, then this process constructs a
localization $Y \to LY$ that also has property $P$. Since
localizations are essentially unique, any localization automatically
has property $P$ as well.

\begin{rmk}
  If our category $\mathcal{C}$ does not have enough colimits, the
  small object argument may not apply. However, Bousfield
  localizations may still exist even if this particular construction
  cannot be applied.
\end{rmk}

\section{Unstable settings}

The classical examples of Bousfield localization are localizations of
spaces. It is worthwhile first relating the localization condition to
based mapping spaces.
\begin{prop}
  Suppose $f\co A \to B$ is a map of well-pointed spaces with
  basepoint. Then a space $Y$ is $f$-local in the category of
  \emph{unbased} spaces if and only if, for all basepoints $y \in Y$,
  the restriction
  \[
    f^*\co \Map_*(B,Y) \to \Map_*(A,Y)
  \]
  of based mapping spaces is a weak equivalence.
\end{prop}

\begin{proof}
  Evaluation at the basepoint gives a map of fibration sequences
  \[
    \xymatrix{
      \Map_*(B,Y) \ar[d] \ar[r] & \Map(B,Y) \ar[d] \ar[r] & Y \ar@{=}[d] \\
      \Map_*(A,Y) \ar[r] & \Map(A,Y) \ar[r] & Y.
    }
  \]
  The center vertical map is an isomorphism on $\pi_*$ at any
  basepoint if and only if the left-hand map is.
\end{proof}


\begin{rmk}
  As $S$-equivalences are preserved under homotopy pushouts and the
  2-out-of-3 axiom, we find that any space $Y$ local with respect to
  $f\co A \to B$ is also local with respect to the map $B/A \to *$
  from the homotopy cofiber to a point, and thus that every path
  component of $Y$ has a contractible space of based maps $B/A \to
  Y$. However, we will see shortly that the converse does not hold in
  general.
\end{rmk}

\begin{exam}
  Let $\mathcal{S}$ be the category of spaces, and take $f$ to be the
  map $S^n \to *$. Then a space $X$ is $f$-local if and only if, for
  any basepoint $x \in X$, the iterated loop space $\Omega^nX$ at $x$
  is weakly contractible. Equivalently, for $n \geq 1$ the space $X$
  is $f$-local if and only if it is \emph{$(n-1)$-truncated}:
  $\pi_k(X,x)$ is trivial for all $k \geq n$ and all $x \in X$.

  A map $A \to B$ of CW-complexes, by obstruction theory, is an
  $f$-equivalence if and only if it is $(n-1)$-connected. Therefore,
  for $n > 0$ a map $A \to B$ of CW-complexes is an $f$-localization
  if and only if $\pi_k(A) \to \pi_k(B)$ is an isomorphism for
  $0 \leq k < n$ and all basepoints, but $\pi_k B$ vanishes for all
  $k \geq n$ and all basepoints.\footnote{We should be careful about
    edge cases. When $n=0$, $X$ is $(-1)$-truncated if and only if it
    is either empty or weakly contractible. By convention,
    $S^{-1} = \emptyset$, and $X$ is $(-2)$-truncated if and only if
    it is weakly contractible.

    When $n=0$ a map $A \to B$ is an $f$-equivalence if and only if
    either both $A$ and $B$ are empty or neither of them is, and a map
    $A \to X$ is an $f$-localization if and only if either $A$ is
    nonempty and $X$ is weakly contractible, or $A$ and $X$ are both
    empty. When $n=-1$ any map is an $f$-equivalence, and a map
    $A \to X$ is an $f$-localization if and only if $X$ is weakly
    contractible.} This characterizes a stage $P_{n-1}(X)$ in the
  Postnikov tower of $X$.
\end{exam}

\begin{exam}
  Let $f$ be the inclusion $S^n \vee S^m \to S^n \times S^m$ of
  spaces. The Cartesian product is formed by attaching an $(n+m)$-cell
  to $S^n \vee S^m$ along an attaching map given by a Whitehead
  product $[\iota_n,\iota_m] \in \pi_{n+m-1}(S^n \vee S^m)$. Any map
  $S^n \vee S^m \to X$, classifying a pair of elements
  $\alpha \in \pi_n(X)$ and $\beta \in \pi_m(X)$ at some basepoint
  $x$, sends this attaching map to $[\alpha,\beta]$. The fiber of
  $\Map(S^n \times S^m, X) \to \Map(S^n \vee S^m, X)$ over the
  corresponding point is either empty (if $[\alpha,\beta]$ is
  nontrivial) or equivalent to the iterated loop space
  $\Omega^{n+m} X$ at $x$ (if $[\alpha,\beta]$ is trivial). A space
  $X$ is therefore local with respect to $f$ if and only if, at any
  basepoint, the homotopy groups $\pi_k(X)$ are zero for all
  $k \geq n+m$ and the Whitehead products
  \[
    \pi_n(X,x) \times \pi_m(X,x) \to \pi_{n+m-1}(X,x)
  \]
  vanish at any basepoint $x$.

  Consider the case $n=m=1$. For a path-connected CW-complex $X$
  with fundamental group $G$, the map $X \to K(G_{ab},1)$ is an
  $f$-localization.
\end{exam}

\begin{exam}
  If $A$ is nonempty, then a space $Y$ is local with respect to
  $f\co \emptyset \to A$ if and only if $Y$ is weakly contractible. All
  maps are $f$-equivalences, and $X \to *$ is always an
  $f$-localization.
\end{exam}

\begin{exam}
  Consider a degree-$p$ map $f\co S^1 \to S^1$. A space $Y$ is
  $f$-local if and only if it is local for degree-$p$ maps $S^n \to
  S^n$, and this occurs if and only if the $p$'th power maps $\pi_n(Y)
  \to \pi_n(Y)$ are all isomorphisms.

  By contrast, let $M(\mb Z/p,1)$ be the Moore space constructed as
  the cofiber of $f$, and consider the map $g\co M(\mb Z/p,1) \to *$.
  A space $Y$ is $g$-local if and only if it satisfies the extension
  condition for the maps $M(\mb Z/p,n) \to *$ for all $n \geq 1$, or
  equivalently if the mod-$p$ homotopy sets $\pi_n(Y;\mb Z/p)$
  vanish for all $n \geq 2$. This is equivalent to the
  $p$'th-power maps being isomorphisms on $\pi_n(Y)$ for all $n > 1$ 
  and injective on $\pi_1(Y)$.
\end{exam}

\begin{exam}[{\cite{neisendorfer-book}}]
  Let $S$ be the set of maps $\{ K(\mb Z/p,1) \to * \}$ as $p$ ranges
  over the prime numbers. Then the Sullivan conjecture, as proven by
  Miller \cite{miller-sullivanconjecture}, is equivalent to the
  statement that any finite CW-complex $X$ is $S$-local. Since
  $S$-equivalences are closed under products and homotopy colimits,
  the expression of $K(\mb Z/p, n+1)$ as the geometric realization of
  the bar construction $\{K(\mb Z/p, n)^q\}$ shows inductively that
  the maps $K(\mb Z/p,n) \to *$ are all $S$-equivalences. However, if
  $Y$ is any nontrivial 1-connected space with finitely generated
  homotopy groups and a finite Postnikov tower, then $Y$ accepts a
  nontrivial map from some $K(\mb Z/p, n)$ and hence cannot be
  $S$-local. This argument shows that a simply-connected finite
  CW-complex with nonzero mod-$p$ homology has $p$-torsion in
  infinitely many nonzero homotopy groups, which was conjectured by
  Serre in the early 1950's and proven by McGibbon and Neisendorfer
  \cite{mcgibbon-neisendorfer-serreconjecture}.
\end{exam}

Localization still applies to other categories closely related to
topological spaces.

\begin{exam}
  Let $\mathcal{C}$ be the category of based spaces. A based space $Y$
  is local with respect to the based map $* \to S^1$ if and only if
  the loop space $\Omega Y$ is weakly contractible, or equivalently if
  and only if the path component $Y_0$ of the basepoint is weakly
  contractible. A model for the Bousfield localization is given by the
  mapping cone of the map $Y_0 \to Y$.
\end{exam}

\begin{exam}
  Fix a discrete group $G$, and consider the category of $G$-spaces:
  spaces with a continuous action of a group $G$, with maps being
  continuous maps. For example, the empty space has a unique
  $G$-action, while the orbit spaces $G/H$ have continuous actions
  under the discrete topology. Every $G$-space has fixed-point
  subspaces $X^H \cong \Map_G(G/H,X)$ for subgroups $H$ of $G$. In
  this context, there is an abundance of examples of localizations.
  
  A $G$-space $Y$ is local with respect to $\emptyset \to *$ if and only
  if the fixed-point subspace $Y^G$ is contractible. A model for
  Bousfield localization is given by the mapping cone of the map $Y^G
  \to Y$.

  Fix a model for the universal contractible $G$-space $EG$. A
  $G$-space $Y$ is local with respect to $EG \to *$ if and only if the
  map from the fixed point space $Y^G$ to the homotopy fixed point
  space $\Map_G(EG,Y) = Y^{hG}$ is a weak equivalence. Since there is a
  $G$-equivariant homotopy equivalence $EG \times EG \to EG$, a model
  for the Bousfield localization is the space of nonequivariant maps
  $\Map(EG,Y)$, with $G$ acting by conjugation.

  A $G$-space $Y$ is local with respect to $\emptyset \to G$ if and only
  if the underlying space $Y$ is contractible. A model for the
  Bousfield localization is given by the mapping cone of the map $EG
  \times Y \to Y$, sometimes called $\widetilde{EG} \wedge Y_+$.
\end{exam}

\begin{exam}
  Fix a collection $S$ of maps and a space $Z$, letting $\mathcal{C}$
  be the category of spaces over $Z$. We say that a map $X \to Y$ of
  spaces over $Z$ is a \emph{fiberwise $S$-equivalence} if the map of
  homotopy fibers over any point $z \in Z$ is an $S$-equivalence, and
  refer to the corresponding localizations as \emph{fiberwise
    $S$-localizations}.

  A map $X \to Y$ over $Z$ which is a weak equivalence on underlying
  spaces is in particular a fiberwise $S$-equivalence. Applying this
  to the lifting characterization of fibrations, we can find that for
  an object $Y \to Z$ of $\mathcal{C}$ to be fiberwise $S$-local the
  map $Y \to Z$ must be a fibration. Moreover, for fibrations
  $Y \to Z$ we can recharacterize being local. Given any map
  $f\co A \to B$ in $S$ and any point $z \in Z$, there is a map in
  $\mathcal{C}$ of the form $f_z\co A \to B \to \{z\} \subset Z$
  concentrated entirely over the point $z$; let $S_Z$ be the set of
  all such maps. A fibration $Y \to Z$ in $\mathcal{C}$ fiberwise
  $S$-local if and only if it is $S_Z$-local in $\mathcal{C}$.

  Fiberwise localizations were constructed by Farjoun in
  \cite[1.F.3]{farjoun-localization}; they are also constructed in
  \cite[\S 7]{hirschhorn} and characterized from several perspectives.
\end{exam}

\begin{exam}
  The category of topological monoids and continuous homomorphisms has
  its own homotopy theory. Consider the inclusion
  $f\co \mb N \to \mb Z$ of discrete monoids. Then
  $\Map^{mon}(\mb Z, M) \to \Map^{mon}(\mb N, M)$ is isomorphic to the
  map $M^\times \to M$ from the space of invertible elements of $M$ to
  the space $M$.\footnote{As a point-set digression the reader should,
    as usual, be warned that the source may not have the subspace
    topology. The space of invertible elements is, instead,
    homeomorphic to the subspace of $M \times M$ of pairs of elements
    $(x,y)$ such that $xy = yx = 1$.}  An $f$-local object is a
  topological group, and localization is a topologized version of
  group-completion.

  We note, however, that the map $\mb N \to \mb Z$ does not
  participate well with \emph{weak} equivalences of topological
  monoids: weakly equivalent topological monoids do not have weakly
  equivalent spaces of invertible elements because homomorphisms out
  of $\mb Z$ are not homotopical. We can get a version that respects
  weak equivalences in two ways. With model categories, we can factor
  the map $\mb N \to \mb Z$ as
  $\mb N \hookrightarrow \mb Z_c \too{\simeq} \mb Z$ in the category
  of topological monoids, where $\mb Z_c$ is a cofibrant topological
  monoid, and there are explicit models for such. We could instead use
  coherent multiplications, where a map $\mb Z \to M$ is no longer
  required to strictly be a homomorphism but instead be a coherently
  multiplicative map.

  Using either correction, the space $M^\times$ of strict units
  becomes replaced, up to equivalence, by the pullback
  \[
    \xymatrix{
      M^{inv} \ar[r] \ar[d] & M \ar[d] \\
      \pi_0(M)^\times \ar[r] & \pi_0 M,
    }
  \]
  the union of the components of $M$ whose image in $\pi_0(M)$ has an
  inverse. A local object is then a \emph{grouplike} topological
  monoid, and localization is homotopy-theoretic group
  completion. These play a key role the study of iterated loop spaces
  and algebraic $K$-theory \cite{may-loopspaces,
    segal-categoriescohomology, mcduff-segal-groupcompletion}.
\end{exam}

\section{Stable settings}

One of the great benefits of the stable homotopy category, and stable
settings in general, is that a map $f\co X \to Y$ becoming an
equivalence is roughly the same as the cofiber $Y/X$ becoming
trivial.

We recall the definition of stability from \cite[\S
1.1.1]{lurie-higheralgebra}.

\begin{defn}
  The category $\mathcal{C}$ is \emph{stable} if it satisfies the
  following properties:
  \begin{enumerate}
  \item $\mathcal{C}$ is (homotopically) \emph{pointed}: there is an
    object $*$ such that, for all $X \in \mathcal{C}$, the spaces
    $\Map_{\mathcal{C}}(X,*)$ and $\Map_{\mathcal{C}}(*,X)$ are
    contractible.
  \item $\mathcal{C}$ has homotopy pushouts of diagrams $*
    \leftarrow X \to Y$ and homotopy pullbacks of diagrams $* \to Y
    \leftarrow X$.

    As a special case, we have suspension and loop objects:
    \begin{align*}
      \Sigma X &= \hocolim(* \leftarrow X \to *)\\
      \Omega X &= \holim(* \to X \leftarrow *)
    \end{align*}
  \item Suppose that we have a homotopy coherent diagram
    \[
      \xymatrix{
        X \ar[r] \ar[d] & Y \ar[d]\\
        \ast \ar[r] & Z,
      }
    \]
    meaning maps as given and a homotopy between the double
    composites. Then the induced map
    \[
      \hocolim (* \leftarrow X \to Y) \to Z
    \]
    is a homotopy equivalence if and only if the map
    \[
      X \to \holim (* \to Z \leftarrow Y)
    \]
    is a homotopy equivalence.

    Taking $Y = \ast$, we find that a map
    $X \to \Omega Z$ is an equivalence if and only if the
    homotopical adjoint $\Sigma X \to Z$ is an equivalence.
  \end{enumerate}
\end{defn}

\begin{exam}
  The category of (cofibrant--fibrant) spectra is the canonical
  example of a stable category.
\end{exam}

\begin{exam}
  For any ring $R$, there is a category $\mathcal{K}_R$ of chain complexes of
  $R$-modules. Any two complexes $C$ and $D$ have a Hom-complex
  $\Hom_R(C,D)$, and the Dold--Kan correspondence produces a
  simplicial set $\Map_{\mathcal{K}_R}(C,D)$ whose homotopy groups satisfy
  \[
    \pi_n \Map_{\mathcal{K}_R}(C,D) \cong H_n \Hom_R(C,D)
  \]
  for $n \geq 0$.\footnote{More generally, if $R[m]$ is the complex
    equal to $R$ in degree $n$ and zero elsewhere, then for all
    complexes $C$ we have $[R[m], C]_{h\mathcal{K}_R} \cong H_m(C)$.}
  This gives the category $\mathcal{K}_R$ of complexes an enrichment
  in simplicial sets, and these mapping spaces make the category
  $\mathcal{K}_R$ stable. Within this category there are many stable
  subcategories: categories of complexes which are bounded above or
  below or both, with homology groups bounded above or below or both,
  which are made up of projectives or injectives, and so on.

  We will write $\mathcal{C}_R$ be the category of cofibrant
  objects in the projective model structure on $R$, whose homotopy
  category is the derived category $D(R)$.
\end{exam}

\begin{thm}[{see \cite[Theorem~1.1.2.14]{lurie-higheralgebra}}]
  If $\mathcal{C}$ is stable, then the homotopy category
  $h\mathcal{C}$ has the structure of a triangulated category.
\end{thm}

In a stable category, every object $Y$ has an equivalence $Y \to
\Omega \Sigma Y$. However, there is a natural weak equivalence
\begin{align*}
  \Map_{\mathcal{C}}(X, \Omega Z)
  &\simeq \holim \left[\Map_{\mathcal{C}}(X, *) \to
    \Map_{\mathcal{C}}(X,Z) \leftarrow
    \Map_{\mathcal{C}}(X, *)\right]\\
  &\simeq \holim(* \to \Map_{\mathcal{C}}(X,Z) \leftarrow *)\\
  &\simeq \Omega \Map_{\mathcal{C}}(X,Z),
\end{align*}
and hence the mapping spaces
\[
  \Map_{\mathcal{C}}(X,Y) \simeq \Omega^n \Map_{\mathcal{C}}(X,
  \Sigma^n Y)
\]
can be extended to be valued in $\Omega$-spectra. This makes it much
easier to detect equivalences: we only need to check the homotopy
groups of $\Omega^t \Map_{\mathcal{C}}(X,Y)$ at the basepoint.

\begin{defn}
  Suppose that $\mathcal{C}$ is stable and $S$ is a class of maps in
  $\mathcal{C}$. We say that $S$ is \emph{shift-stable} if the image
  $\bar S$ in $h\mathcal{C}$ is closed under suspension and
  desuspension, up to isomorphism.
\end{defn}

\begin{prop}
  Suppose that $\mathcal{C}$ is stable and $S$ is a shift-stable class
  of maps $\{f_\alpha\co A_\alpha \to B_\alpha\}$. Then an object $Y$
  in $\mathcal{C}$ is $S$-local if and only if the homotopy classes of
  maps $[B_\alpha / A_\alpha, X]_{h\mathcal{C}}$ are trivial.
\end{prop}

\begin{proof}
  The individual fiber sequences
  \[
    \Omega^t \Map_{\mathcal{C}}(B_\alpha/A_\alpha, Y) \to 
    \Omega^t \Map_{\mathcal{C}}(B_\alpha, Y) \to 
    \Omega^t \Map_{\mathcal{C}}(A_\alpha, Y),
  \]
  on homotopy classes classes of maps, are part of a long exact
  sequence
  \[
    \dots \to [\Sigma^t B_\alpha / A_\alpha, Y]_{h\mathcal{C}}
    \to \pi_t \Map_{\mathcal{C}}(B_\alpha, Y)
    \to \pi_t \Map_{\mathcal{C}}(A_\alpha, Y)
    \to [\Sigma^{t-1} B_\alpha / A_\alpha, Y]_{h\mathcal{C}}
    \to \dots
  \]
  from the triangulated structure. We get an isomorphism on homotopy
  groups if and only if the terms
  $[\Sigma^t B_\alpha / A_\alpha, Y]_{h\mathcal{C}}$ vanish for all
  values of $t$.
\end{proof}

By contrast with the unstable case where basepoints are a continual
issue, these shift-stable localizations in a stable category are
always nullifications, and they are \emph{equivalent} to
nullifications of the triangulated homotopy category by a class $S$
that is closed under shift operations.

\begin{defn}
  Suppose that $\mathcal{D}$ is a triangulated category. A full
  subcategory $\mathcal{T}$ is called a \emph{thick subcategory} if
  its objects are closed under closed under isomorphism, shifts,
  cofibers, and retracts. If $\mathcal{D}$ has coproducts, a thick
  subcategory $\mathcal{T}$ is \emph{localizing} if it is also closed
  under coproducts.
\end{defn}

\begin{prop}
  Suppose that $\mathcal{D}$ is a triangulated category and that
  $\mathcal{T} \subset \mathcal{D}$ is a thick subcategory. Then there
  exists a triangulated category $\mathcal{D}/\mathcal{T}$ called the
  \emph{Verdier quotient} of $\mathcal{D}$ by $\mathcal{T}$, with a
  functor $\mathcal{D} \to \mathcal{D} / \mathcal{T}$. The Verdier
  quotient is universal among triangulated categories under
  $\mathcal{D}$ such that the objects of $\mathcal{T}$ map to trivial
  objects.
\end{prop}

This universal characterization allows us to strongly relate Bousfield
localization of stable categories to localization of the homotopy
category.

\begin{prop}
  Suppose that $\mathcal{C}$ is stable, and that $S$ is a shift-stable
  collection of maps in $\mathcal{C}$.
  \begin{enumerate}
  \item An object in $\mathcal{C}$ is $S$-local if and only if its
    image in the homotopy category $h\mathcal{C}$ is $S$-local.
  \item A map in $\mathcal{C}$ is an $S$-equivalence if and only if
    its image in the homotopy category is an $S$-equivalence.
  \item The subcategories $L^S \mathcal{C}$ of $S$-local objects and
    $\mathcal{T}$ of $S$-trivial objects are thick subcategories of
    $\mathcal{C}$.
  \item The subcategory $\mathcal{T}$ of $S$-trivial objects is closed
    under all coproducts that exist in $\mathcal{C}$. If $\mathcal{C}$
    has small coproducts then it is a localizing subcategory.
  \item If all objects in $\mathcal{C}$ have $S$-localizations, then
    the left adjoint to the inclusion $ hL^S \mathcal{C} \to
    h\mathcal{C}$ has a factorization 
    \[
      h\mathcal{C} \to h\mathcal{C} / h\mathcal{T} \to h L^S \mathcal{C}.
    \]
    The latter functor is an equivalence of categories.
  \end{enumerate}
\end{prop}

\begin{rmk}
  The fact that Bousfield localization of $\mathcal{C}$ is determined
  by a construction purely in terms of $h\mathcal{C}$ is special to
  the stable setting.
\end{rmk}

\begin{rmk}
  This relates Verdier quotients in a stable category to Bousfield
  localization, but only quotients by a \emph{localizing}
  subcategory. For a homotopical interpretation of more general
  Verdier quotients, see \cite[\S I.3]{scholze-nikolaus-tc}.
\end{rmk}

%


\begin{exam}
  Let $S$ be the collection of multiplication-by-$m$
  maps $S^n \to S^n$ for $n \in \mb Z$, $m > 0$. A spectrum $Y$
  is $S$-local if and only if multiplication by $m$ is an isomorphism
  on the homotopy groups $\pi_* Y$ for all positive $m$, or
  equivalently if the maps $\pi_* Y \to \mb Q \otimes \pi_*Y$ are
  isomorphisms. Such spectra are called \emph{rational}.

  If $Y$ is such a spectrum, we can calculate that the natural map
  \[
    [X,Y] \to \prod_n \Hom(\pi_n X, \pi_n Y)
  \]
  is an isomorphism for any spectrum $X$: because $\pi_n Y$ is a
  graded vector space, $\Hom(-, \pi_n Y)$ is exact and so both sides
  are cohomology theories in $X$ that satisfy the wedge axiom and
  agree on spheres. Therefore, $A \to B$ is an $S$-equivalence if and
  only if $\mb Q \otimes \pi_n(A) \to \mb Q \otimes \pi_n(B)$ is an
  isomorphism for all $n$, and such maps are called \emph{rational
    equivalences}. In this case, this is the same as the map
  $H_*(A;\mb Q) \to H_*(B;\mb Q)$ being an isomorphism.

  This analysis allows us to conclude that $X \to H\mb Q \wedge X =
  X_{\mb Q}$ is a rationalization for all $X$.
\end{exam}

\begin{exam}
  In the above, we can make $S$ smaller. If $S$ is the set of
  multiplication-by-$p$ maps $S^n \to S^n$, we similarly find that
  $S$-local spectra are those whose homotopy groups are
  $\mb Z[1/p]$-modules, and that equivalences are those maps which
  induce isomorphisms on homotopy groups after inverting $p$. The
  localization of $\mb S$ is the homotopy colimit
  \[
    \mb S[1/p] = \hocolim (\mb S \too{p} \mb S \too{p} \mb S \too{p} \dots),
  \]
  which is also a Moore spectrum for $\mb Z[1/p]$. We similarly find
  that $X \to \mb S[1/p] \wedge X$ is an $S$-localization for all $X$.

  We could also let $S$ be the set of multiplication-by-$m$ maps for
  $m$ relatively prime to $p$, which replaces the ring $\mb Z[1/p]$
  with the local ring $\mb Z_{(p)}$ in the above.
\end{exam}

\begin{exam}
  Fix a commutative ring $R$ and a multiplicatively closed subset
  $W \subset R$, recalling that localization with respect to $W$ is
  exact. If we define $S$ to be the set of maps of the form
  $R[n] \too{w} R[n]$ for $w \in W$, then a complex $C$ of $R$-modules
  is $S$-local if and only if the multiplication-by-$w$ maps
  $H_*(C) \to H_*(C)$ are isomorphisms, or equivalently if and only if
  $H_*(C) \to W^{-1} H_*(C) \cong H_*(W^{-1} C)$ is an isomorphism. A
  map $A \to B$ of complexes is an $S$-equivalence if and only if the
  map $W^{-1} A \to W^{-1} B$ is an equivalence.

  The natural map $C \to W^{-1} C \cong W^{-1} R \otimes_R C$ is an
  $S$-localization.
\end{exam}

These examples have such nice properties that it is convenient to
axiomatize them.

\begin{defn}
  A stable Bousfield localization on spectra\footnote{This definition
    extends if we have a stable category $\mathcal{C}$ with a
    symmetric monoidal structure appropriately compatible with the
    stable structure.} is a \emph{smashing localization} if either of
  the following equivalent conditions hold.
  \begin{enumerate}
  \item There is a map of spectra $\mb S \to L\mb S$ such that, for
    any $X$, the map $X \to L\mb S \wedge X$ is a localization.
  \item Local objects are closed under arbitrary homotopy colimits.
  \end{enumerate}
\end{defn}

The equivalence between these two characterizations is not immediately
obvious. The first implies the second, because
\[
  L\mb S \wedge \hocolim X_i \to \hocolim(L\mb S \wedge X_i)
\]
is always an equivalence and the former is always local. The converse
follows because the only homotopy-colimit preserving functors on
spectra are all equivalent to functors of the form
$X \mapsto A \wedge X$ for some $A$, and the resulting localization
map $\mb S \to A$ is of the desired form.

\begin{exam}
  \label{exam:p-completion}
  A spectrum $Y$ is local for the maps $\mb S[1/p] \wedge S^n \to *$ if
  and only if the homotopy limit
  \[
    \holim (\dots \to Y \too{p} Y \too{p} Y) \simeq F(\mb S[1/p],Y)
  \]
  of function spectra is weakly contractible. However, taking homotopy
  limits of the natural fiber sequences
  \[
    \xymatrix{
      \dots \ar[r] & Y \ar[r]^p \ar[d]_{p^2} & Y \ar[r]^p \ar[d]_p & Y \ar[d]_1 \\
      \dots \ar[r] & Y \ar[r]^1 \ar[d] & Y \ar[r]^1 \ar[d] & Y \ar[d] \\
      \dots \ar[r] & Y /p^2 \ar[r] & Y / p \ar[r]  & \ast
    }
  \]
  shows that $Y$ is local if and only if the map
  $Y \to Y^\wedge_p = \holim Y/p^k$ is an equivalence. Therefore, we
  refer to a spectrum local for these maps as \emph{$p$-complete}; a
  Bousfield localization of $Y$ will be called the
  \emph{$p$-completion}; a trivial object is called \emph{$p$-adically
    trivial}; an equivalence is called a \emph{$p$-adic
    equivalence}. The above presents $Y^\wedge_p$ as a candidate for
  the $p$-completion of $Y$.

  If we construct the fiber sequence
  \[
    \Sigma^{-1} \mb S/p^\infty \to \mb S \to \mb S[1/p],
  \]
  we find that we can identify $Y^\wedge_p$ with the function spectrum
  $F(\Sigma^{-1} \mb S/p^\infty, Y)$. Moreover, the map
  $Y^\wedge_p \to (Y^\wedge_p)^\wedge_p$ is always an
  equivalence. Therefore, $Y^\wedge_p$ is always $p$-complete.

  If multiplication-by-$p$ is an equivalence on $Z$, then $Z \simeq Z
  \wedge \mb S[1/p]$, and so maps $Z \to Y$ are equivalent to maps $Z
  \to F(\mb S[1/p],Y)$. For any $Y$ which is $p$-adically complete,
  this is trivial, so such objects $Z$ are $p$-adically trivial. In
  particular, the fiber of $Y \to Y^\wedge_p$ is always trivial and so
  $Y \to Y^\wedge_p$ is a $p$-adic equivalence. Therefore, this is a
  $p$-adic completion.

  If each homotopy group of $Y$ has a bound on the order of $p$-power
  torsion, we can further identify the homotopy groups of $Y^\wedge_p$
  as the ordinary $p$-adic completions of the homotopy groups of $Y$;
  if the homotopy groups of $Y$ are finitely generated, then
  $\pi_*(Y^\wedge_p) \to \pi_*(Y) \otimes \mb Z_p$.\footnote{In
    general, the homotopy groups of the $p$-adic completion are
    somewhat sensitive and one needs to be careful about derived
    functors of completion.}
\end{exam}

\begin{rmk}
  Note that the previous example is not a smashing localization. For
  any connective spectrum $X$, the map $\mb S^\wedge_p \wedge X \to
  X^\wedge_p$ induces the map $\pi_* (X) \otimes \mb Z_p \to
  \pi_*(X)^\wedge_p$ on homotopy groups; this is typically only an
  isomorphism if the homotopy groups $\pi_*(X)$ are finitely generated.
\end{rmk}

\begin{exam}
  \label{exam:derived-complete}
  For an element $x$ in a commutative ring $R$, let $K_x$ be the complex
  \[
    \dots \to 0 \to R \to x^{-1} R \to 0 \to \dots
  \]
  concentrated in degrees $0$ and $-1$, with a map $K_X \to R$. For a
  sequence of elements $(x_1,\dots,x_n)$, let
  $K_{(x_1,\dots,x_n)} = \bigotimes_R K_{x_i}$ be the \emph{stable Koszul
    complex}. If $y$ is in the ideal generated by $(x_1,\dots,x_n)$,
  then the inclusion $K_{(x_1,\dots,x_n)} \to K_{(x_1,\dots,x_n,y)}$
  is a quasi-isomorphism, and so up to quasi-isomorphism the Koszul
  complex only depends on the ideal. Let $K_I$ be a cofibrant
  replacement.

  We say that a complex $C$ is \emph{$I$-complete} if and only if it
  is local with respect to the shifts of the map $K_I \to R$. This is
  true if and only if the homology groups of $C$ are $I$-complete in
  the derived sense. If $R$ is Noetherian and the homology groups of
  $C$ are finitely generated, this is true if and only if the homology
  groups of $C$ are $I$-adically complete in the ordinary sense.
\end{exam}

These frameworks for the study of localization and completion, and many
generalizations of it, were developed by Greenlees and May
\cite{greenlees-may}.

\begin{exam}
  Fix a ring $R$, and let $\mathcal{C}$ be the category of unbounded
  complexes of finitely generated projective left $R$-modules that
  only have nonzero homology groups in finitely many degrees. Consider
  the set $S$ of maps $R[n] \to 0$. An object $C$ is $S$-local if and
  only if its homology groups are trivial.

  We can inductively take mapping cones of maps $R[n] \to C$ to
  construct a localization $C \to LC$, embedding $C$ into an unbounded
  complex of finitely generated projective modules with trivial
  homology groups. Therefore, localizations exist in this category.

  For two such complexes $C$ and $D$ with trivial homology, we have
  \[
    \Hom_{h\mathcal{C}}(C,D) \cong \lim_n \Hom_R(Z_{n} C, Z_{n} D) / \Hom_R(Z_{n}
    C, D_{n+1})
  \]
  where $D_{n+1} \to Z_{n}(D)$ is the boundary map---a surjective map
  from a projective module.

  This can be interpreted in terms of the stable module category of
  $R$. Defining $W_n(C) = Z_{-n}(C)$, the short exact sequences
  $0 \to Z_{-n}(C) \to C_{-n} \to Z_{-n-1}(C) \to 0$ determine
  isomorphisms $W_n(C) \cong \Omega W_{n+1}(C)$ in the stable module
  category, assembling the $W_n$ into an ``$\Omega$-spectrum''. Maps
  $C \to D$ are then equivalent to maps of $\Omega$-spectra in the
  stable module category.\footnote{In certain cases, such as for
    Frobenius algebras, $\Omega$ is an autoequivalence. This
    definition then simply recovers the stable module category of $R$
    by itself. If $R$ has finite projective dimension,
    $\Omega$-spectrum objects are necessarily trivial.}
\end{exam}

\section{Homology localizations}

\subsection{Homology localization of spaces}
\begin{defn}
  Suppose $E_*$ is a homology theory on spaces. Then we say that a map
  $f\co A \to B$ of spaces is an \emph{$E_*$-equivalence} if it
  induces an isomorphism $f_*\co E_* A \to E_* B$. A space is
  \emph{$E_*$-local} if it is local with respect to the class of
  $E_*$-equivalences.
\end{defn}

\begin{exam}
  Suppose that $E_*$ is integral homology $H_*$. Any Eilenberg--Mac
  Lane space $K(A,n)$ is $H_*$-local by the universal coefficient
  theorem for cohomology. Moreover, any simply-connected space $X$ is
  the homotopy limit of a Postnikov tower built from fibration
  sequences $P_n X \to P_{n-1} X \to K(\pi_n X, n+1)$. Since local
  objects are closed under homotopy limits, we find that
  simply-connected spaces are $H_*$-local.\footnote{This argument can
    be refined to show that \emph{nilpotent spaces} (where $\pi_1(X)$
    is nilpotent, and acts nilpotently on the higher homotopy groups)
    are $H_*$-local.}
\end{exam}

\begin{rmk}
  This example illustrates a very different approach to the
  construction of localizations. Because homology isomorphisms are
  detected by the $K(A,n)$, these spaces are automatically local;
  therefore, any object built from these using homotopy limits is
  automatically local. Such objects are often called
  \emph{nilpotent}. Thus gives us a dual approach to building the
  Bousfield localization of $X$: construct a natural diagram of
  nilpotent objects that receive maps from $X$, and try to verify that
  the homotopy limit is a localization of $X$.
\end{rmk}

\begin{exam}
  Serre's rational Hurewicz theorem implies that a map of
  simply-connected spaces is an isomorphism on rational homology
  groups if and only if it is an isomorphism on rational homotopy
  groups. A simply-connected space is local for rational
  homology if and only if it its homotopy groups are rational vector
  spaces.

  The same is not true for general spaces. The map $\mb{RP}^2 \to *$
  is a rational homology isomorphism, and the covering map
  $S^2 \to \mb{RP}^2$ is an isomorphism on rational homotopy groups,
  but the composite $S^2 \to *$ is neither. The problem here is the
  failure of a simple Postnikov tower for $\mb{RP}^2$ due to the
  action of $\pi_1$ on the higher homotopy groups.
\end{exam}

\begin{exam}
  If $X$ is a connected space with perfect fundamental group, then
  Quillen's plus-construction gives a map $X \to X^+$ that
  induces an $H_*$-isomorphism such that $X^+$ is
  simply-connected. This makes $X^+$ into an $H_*$-localization of
  $X$.
\end{exam}

Classically, Quillen's plus-construction can be applied to groups with
a perfect subgroup. In order to properly identify the universal
property, we need to work in a relative situation.

\begin{exam}
  Fix a group $G$, and let $\mathcal{C}$ be the category of spaces
  over $BG$. Given an abelian group $A$ with $G$-action, there is an
  associated local coefficient system $\underline A$ on $BG$, and so
  given any object $X \to BG$ of $\mathcal{C}$ we can define the
  homology groups $H_*(X; \underline A)$. We say that a map $X \to Y$
  over $BG$ is a relative homology equivalence if it induces
  isomorphisms on homology with coefficients in any $\underline
  A$. Taking $A$ to be the group algebra $\mb Z[G]$, we find that this
  is equivalent to the map of homotopy fibers $F_X \to F_Y$ being a
  homology isomorphism, so this is the same as a \emph{fiberwise
    $H_*$-equivalence}. If an object $Y$ over $BG$ has
  simply-connected homotopy fiber it is automatically
  local.

  Suppose that $X$ is any connected space such that $\pi_1(X)$ contains
  a perfect normal subgroup $P$ with quotient group $G$. The
  homomorphism $\pi_1(X) \to G$ lifts to a map $X \to BG$. The
  plus-construction with respect to $P$ is a fiber homology
  equivalence $X \to X^+$ where $X^+ \to BG$ has simply-connected
  homotopy fiber, and thus is a localization in $\mathcal{C}$.
\end{exam}

Localization with respect to homology is very difficult to analyze in
the case when a space is not simply-connected, especially if the space
is not \emph{simple} (either the fundamental group is not nilpotent or
it does not act nilpotently on the higher homotopy groups). Many
natural spaces are not local. Here are some basic tools to prove this.

\begin{lem}
  Suppose that $F_n$ is a free group on $n$ generators and
  $\alpha\co F_n \to F_n$ is a homomorphism, with induced map
  $\alpha_{ab}\co \mb Z^n \to \mb Z^n$. Under the identification
  $\Hom(F_n, G) \cong G^n$ for any group $G$, write $\alpha^*$ for the
  natural map of sets $G^n \to G^n$.

  Suppose the map $\alpha_{ab}$ becomes an isomorphism after tensoring
  with a ring $R$. Then, for any space $X$, a necessary condition for
  $X$ to be $H_*(-;R)$-local is that $\alpha^* \co \pi_1(X,x)^n \to
  \pi_1(X,x)^n$ must be a bijection at any basepoint.
\end{lem}

\begin{proof}
  The map $\alpha_{ab}$, after tensoring with $R$, can be identified
  with the map $H_1(F_n;R) \to H_1(F_n;R)$ on homology induced by
  $\alpha$. If $\alpha_{ab}$ becomes an isomorphism after tensoring
  with $R$, then $\alpha\co K(F_n, 1) \to K(F_n,1)$ is
  an $H_*(-;R)$-equivalence.

  For a space $X$ to be $H_*(-;R)$-local, the induced map
  \[
    \Map_*(K(F_n,1),X) \to \Map_*(K(F_n,1), X)
  \]
  must be a weak equivalence. Taking a wedge of circles as our model,
  we find that the induced map
  \[
    (\Omega X)^n \to (\Omega X)^n
  \]
  must be a weak equivalence. On $\pi_0$, this is the map $\alpha^*$
  on $\pi_1(X)^n$.
\end{proof}

\begin{exam}
  For $n \neq 0$, the multiplication-by-$n$ map $\mb Z \to \mb Z$
  is a rational isomorphism. Therefore, for $X$ to be rationally
  local, the $n$'th power map $\pi_1(X) \to \pi_1(X)$ should be a
  bijection: every element $g \in \pi_1(X)$ has a unique $n$'th root
  $g^{1/n}$. Such groups are called uniquely divisible, or sometimes
  $\mb Q$-groups. The structure of free $\mb Q$-groups was studied in
  \cite{baumslag-roots}.
\end{exam}

\begin{exam}
  Let $F_2$ be free on the generators $x$ and $y$, and define $\alpha
  \co F_2 \to F_2$ by
  \begin{align*}
    \alpha(x) &= x^{-9} y^{-20} (y^2 x)^{10}\\
    \alpha(y) &= x^{-9} y^{10} (y x^{-1})^{-9}.
  \end{align*}
  The map $\alpha_{ab}$ is the identity map. Therefore, for a space
  with fundamental group $G$ to be local with respect to integral
  homology, any pair of elements $(z,w) \in G$ has to be uniquely of the
  form $(z,w) = (x^{-9} y^{-20} (y^2 x)^{10}, x^{-9} y^{-10} (y x^{-1})^{-9})$
  for some $x$ and $y$ in $G$. Most groups do not satisfy this
  property.

  We can use this to show that any space whose fundamental group $G$
  has a surjective homomorphism $\phi\co G \to A_5$ cannot be local
  with respect to integral homology---in particular, this applies to a
  free group $F_2$. Choose elements $x$ and $y$ in $G$ with
  $\phi(x) = (1 2 3)$ and $\phi(y) = (1 2 3 4 5)$. Then
  $\phi(y^2 x) = (1 4)(2 5)$ and $\phi(y x^{-1}) = (1 4 5)$, and
  $\phi \circ \alpha$ is the trivial homomorphism while $\phi$ is
  surjective.\footnote{In order to use this \emph{particular}
    technique to show that $\phi$ was not a bijection, we needed to
    have a homomorphism $\phi$ whose image was a perfect group---the
    image of $\alpha_{ab}$ is contained in the kernel of
    $\phi_{ab}$. This particular map $\alpha$ is complicated because
    it was reverse-engineered from $\phi$.}

  Several other, more easily defined, maps $\alpha$ can be shown to
  not be bijective. For example, the map
  $(x,y) \mapsto (x [x,y], y [x,y])$ can be shown to not be a
  bijection, e.g. by using Fox's free differential calculus
  \cite{fox-freecalculus1}.
\end{exam}

\begin{lem}
  Let $G$ be a group, $R$ a ring, and $\beta \in \mb Z[G]$ an element
  such that the composite ring homomorphism $\mb Z[G] \too{\epsilon}
  \mb Z \to R$ sends $\beta$ to zero.

  Then, for any based space $X$ with fundamental group $G$, a
  necessary condition for $X$ to be $H_*(-;R)$-local is that
  $\pi_k(X)$ must be complete in the topology defined by
  $\beta$.\footnote{This refers to being \emph{derived} complete in
    the sense of Example~\ref{exam:derived-complete}.}
\end{lem}

\begin{proof}
  Fix the space $X$ and basepoint and consider the space
  $Y = X \vee S^k$. The group $\pi_k(Y)$ is isomorphic to
  $\pi_k(X) \oplus \mb Z[G]$, and so the element $\beta \in \mb Z[G]$
  lifts to a map $\beta\co Y \to Y$ given by the identity on $X$
  together with the map $S^k \to Y$ corresponding to the element
  $(0,\beta) \in \pi_k(X) \oplus \mb Z[G]$. The induced self-map of
  \[
    H_*(Y;R) \cong H_*(X;R) \oplus \widetilde H_*(S^k;R)
  \]
  is given by the identity on $H_*(X;R)$ together with the map
  $\epsilon(\beta)$ tensored with $R$ on the second factor. If
  $\epsilon(\beta)$ becomes zero after tensoring with $R$, then this
  map is zero on the second factor.

  Define
  \[
    X' = \hocolim (Y \too{\beta} Y \too{\beta} \cdots).
  \]
  By construction, the map
  \[
    H_*(X;R) \to H_*(X';R) = \colim H_*(Y;R)
  \]
  is an isomorphism. Therefore, $X \to X'$ is an $H_*(-;R)$-equivalence.

  For $X$ to be $H_*(-;R)$-local, the induced map
  \[
    \Map(X',X) \to \Map(X,X)
  \]
  must be a weak equivalence. Taking the fiber over the identity map
  of $X$, we find that there is an induced equivalence
  \[
    \holim(\cdots \too{\beta} \Omega^k X \too{\beta}
    \Omega^k X) \too{\sim} \ast.
  \]
  Using the Milnor $\lim^1$-sequence, we find that all of the homotopy
  groups of $X$ must be derived-complete with respect to $\beta$.
\end{proof}

\begin{rmk}
  If $R = \mb Z$, then this implies that any element $s \in \mb Z[G]$
  with $\epsilon(s) = \pm 1$ must act invertibly on the higher
  homotopy groups of $X$, and so the action must factor through a
  large localization $S^{-1} \mb Z[G]$.
\end{rmk}

\begin{exam}
  Consider $X = S^1 \vee S^2$, whose fundamental group is isomorphic
  to $\mb Z$ with generator $t$. The second homotopy group satisfies
  \[
    \pi_2(S^1 \vee S^2) \cong \mb Z[t^{\pm 1}]
  \]
  as a module over $\mb Z[t^{\pm 1}]$. This is not complete with
  respect to the ideal generated by $\beta = (t-1)$ even though
  $\epsilon(\beta) = 0$. Therefore, $S^1 \vee S^2$ is not local with
  respect to integral homology.
\end{exam}

\begin{exam}
  The space $\mb{RP}^2$ has fundamental group $\mb Z/2$ generated by
  an element $\sigma$, and the second homotopy group $\mb Z$ satisfies
  $\sigma(y) = -y$. The element $(1 - \sigma)$ has
  $\epsilon(1-\sigma) = 0$ and acts as multiplication by $2$. 
  Since $\mb Z$ is not complete in the $2$-adic topology we find that
  $\mb{RP}^2$ is not local with respect to integral
  homology.\footnote{The homology localization of $\mb{RP}^2$ has, in
    fact, a fiber sequence
    $(S^2)^\wedge_2 \to L\mb{RP}^2 \to K(\mb Z/2,1)$.}
\end{exam}

\begin{exam}
  If $R = \mb Q$, then any element $S \in \mb Z[G]$ with
  $\epsilon(s) \neq 0$ must act invertibly on the higher homotopy
  groups of $X$ for $X$ to be local with respect to rational
  homology. The homotopy groups of $K(\mb Q,1) \vee (S^3)_{\mb Q}$ are
  $\mb Q$ in degree $1$ and the rational group algebra $\mb Q[\mb Q]$
  in degree 3. If $t$ is the generator of $\mb Z \subset \mb Q$, the
  element $2t-1$ has $\epsilon(2t-1) = 1$ and does not act invertibly
  on this group algebra. Therefore, this space is not local with
  respect to rational homology even though its homotopy groups are
  rational.
\end{exam}

\begin{rmk}
  \label{rmk:bousfield-smith}
  Bousfield localization with respect to $E_*$-equivalences leads us
  to some uncomfortable pressure with our previous notation. At first
  glance, it is not clear whether being an equivalence on
  $E_*$-homology is the same as having the same mapping spaces into
  any $E_*$-local object.\footnote{One could, but should not, say it
    this way: it is not clear that an ($E_*$-equivalence)-equivalence
    is automatically an $E_*$-equivalence.} To prove this, one needs
  to prove that there is a sufficient supply of $E_*$-local objects:
  for any $X$, we need to be able to construct an $E_*$-homology
  isomorphism $X \to L_E X$ such that $L_E X$ is $E_*$-local.  Here is
  how Bousfield addressed this in
  \cite[Theorem~11.1]{bousfield-spacelocalization}. It is essentially
  a cardinality argument, whose general form is called the
  Bousfield--Smith cardinality argument in \cite[\S 2.3]{hirschhorn}.
  
  Let $E_*$ be a homology theory on spaces. We then have a class $S$
  of $E_*$-equivalences, which are those maps which induce
  equivalences on $E_*$-homology. Unfortunately, this is a proper
  class of morphisms, and so we cannot immediately apply the small
  object argument to construct localizations. Moreover, because we do
  not know anything about local objects we cannot assert that an
  $S$-equivalence $X \to Y$ is the same as a map inducing an
  isomorphism $E_* X \to E_* Y$.

  Bousfield addresses this by showing the following. Suppose $K \to L$
  is an inclusion of simplicial sets such that $E_* K \to E_* L$ is an
  isomorphism, and that we choose any simplex $\sigma$ of $L$. Then
  there exists a subcomplex $L' \subset L$ with the following
  properties:
  \begin{enumerate}
  \item The simplex $\sigma$ is contained in $L'$.
  \item The map $E_*(K \cap L') \to E_*(L')$ is an isomorphism on $E_*$.
  \item The complex $L'$ has size bounded by a cardinal $\kappa$,
    which \emph{depends only on $E$}.
  \end{enumerate}
  Because of the cardinality bound on $L'$, we can find a \emph{set}
  $T$ of $E_*$-equivalences $A \to B$ so that any such map
  $K \cap L' \to L'$ must be isomorphic to one of them; an arbitrary
  $E_*$-equivalence $K \to L$ can then be factored as a (possibly
  transfinite) sequence of pushouts along the maps in the set $T$
  followed by an equivalence. The maps in $T$ are $E_*$-isomorphisms,
  and an object is $S$-local if and only if it is $T$-local. The small
  object argument then applies to $T$, allowing us to construct
  $T$-localizations $Y \to LY$ which are also $E_*$-isomorphisms.

  We will see in \S~\ref{sec:model-categories} and
  \S~\ref{sec:infty-categories}, in general constructions of Bousfield
  localization, that this verification is the key step.
\end{rmk}

\subsection{Homology localization of spectra}

\begin{defn}
  For a spectrum $E$, a map $f\co X \to Y$ is an \emph{$E$-homology
    equivalence} (or simply an $E$-equivalence) if the corresponding
  map $E_* X \to E_* Y$ is an isomorphism, and we say that $Z$ is
  \emph{$E$-trivial} if $E_* Z = 0$. A map $f$ is an $E$-equivalence
  if and only if the cofiber of $f$ is $E$-trivial.\footnote{Again,
    the definitions of this section can be applied to a stable
    category $\mathcal{C}$ with a compatible symmetric monoidal
    structure.}
\end{defn}

This is most often employed when $E$ is a ring spectrum.
\begin{prop}
  If $E$ has a multiplication $m\co E \wedge E \to E$ with a left unit
  $\eta\co \mb S \to E$ in the homotopy category, then any spectrum
  $Y$ with a unital map $E \wedge Y \to Y$ is $E$-local.
\end{prop}

\begin{rmk}
  Such spectra $Y$ are sometimes called \emph{homotopy $E$-modules.}
  Any spectrum of the form $E \wedge W$ is a homotopy
  $E$-module.
\end{rmk}

\begin{proof}
  Any map $f\co Z \to Y$ has the following factorization in
  the homotopy category:
  \[
    Z \too{\eta \wedge 1} E \wedge Z \too{1 \wedge f} E \wedge Y
    \too{m} Y
  \]
  If $Z$ has trivial $E$-homology, then $E \wedge Z$ is trivial and so
  the composite $Z \to Y$ is nullhomotopic. Therefore, $[Z, Y] = 0$
  for all $E$-trivial $Z$, as desired.
\end{proof}

\begin{cor}
  If $E$ has a multiplication $m\co E \wedge E \to E$ with a left unit
  $\eta\co \mb S \to E$ in the homotopy category, then any homotopy
  limit of  spectra that admit homotopy $E$-module structures is
  $E$-local.
\end{cor}

\begin{exam}
  A particular case of interest is when $E = H\mb Z$. Any
  Eilenberg--Mac Lane spectrum $HA$ is $H\mb Z$-local, being of the
  form $H\mb Z \wedge MA$ for a Moore spectrum for $A$.

  Then any connective spectrum $Y$ is $H\mb Z$-local, as follows.  As
  $H\mb Z$-local objects form a thick subcategory, any spectrum with
  finitely many nonzero homotopy groups is therefore $H\mb
  Z$-local. If $Y$ is connective then $P_n Y$ is $H\mb Z$-local due to
  having a finite Postnikov tower. Therefore, $Y = \holim P_n Y$ is
  the homotopy limit of $H\mb Z$-local spectra, and is thus
  $H\mb Z$-local.

  Similarly, any product of Eilenberg--Mac Lane spectra
  $\prod \Sigma^n HA_n$ is also $H\mb Z$-local. Any rational spectrum
  is of this form.

  However, not all spectra are $H\mb Z$-local. For any prime $p$ and
  integer $n > 0$, there are $p$-primary Morava $K$-theories $K(n)$
  such that $H\mb Z \wedge K(n)$ is trivial; these are
  $H\mb Z$-acyclic. The complex $K$-theory spectrum $KU$ satisfies the
  property that $H_*(KU;\mb Z) \to H_*(KU;\mb Q)$ is an isomorphism:
  from this we can find that $KU \to KU_{\mb Q}$ is an
  $H\mb Z$-equivalence. The target is also $H\mb Z$-local because it
  is rational, and so $KU_{\mb Q}$ is the $H\mb Z$-localization of
  $KU$.
\end{exam}

\begin{exam}
  We can consider the case where $E = H\mb Z/p$. By a similar
  argument, we find that any connective spectrum which is $p$-adically
  complete in the sense of Example~\ref{exam:p-completion} is also
  $H\mb Z/p$-complete. Again, in connective cases there is not a
  difference between being $p$-adically complete and being
  $H\mb Z/p$-local.

  For nonconnective spectra, these are quite different. The Morava
  $K$-theories $K(n)$ are $p$-adically complete but
  $H\mb Z/p$-trivial. The periodic complex $K$-theory spectrum $KU$
  has $\pi_* (KU^\wedge_p) \cong (\pi_* KU)^\wedge_p$, but $KU$ is
  also $H\mb Z/p$-trivial.
\end{exam}

These localizations have the flavor of completion with respect to an
ideal. In some cases we can express them as such.
\begin{defn}
  Suppose that $E$ has a binary multiplication $m$ with a left unit
  $\eta\co \mb S \to E$, and let $j\co I \to \mb S$ be the fiber of
  $\eta\co \mb S \to E$. Assemble these into the inverse system
  \[
    \dots \to I^{\wedge 3}\too{j \wedge 1 \wedge 1} I \wedge I \too{j
      \wedge 1} I \too{j} \mb S
  \]
  The \emph{$E$-nilpotent completion} $X^\wedge_E$ is the homotopy
  limit
  \[
    \holim_n (\mb S / I^{\wedge n}) \wedge X,
  \]
  with map $X \to X^\wedge_E$ induced by the maps $\mb S \to \mb S /
  I^{\wedge n}$.
\end{defn}

\begin{prop}
  The $E$-nilpotent completion is always $E$-local.

  If $E$ is a finite complex, or $X$ and $I$ are connective and $E$ is
  of finite type, then the map $X \to X^\wedge_E$ is an $E$-localization.
\end{prop}

\begin{proof}
  The cofiber sequence $I \to \mb S \to E$, after smashing with
  $I^{\wedge (n-1)}$, becomes a cofiber sequence
  $I^{\wedge n} \to I^{\wedge (n-1)} \to E \wedge I^{\wedge (n-1)}$, and so there are
  cofiber sequences
  \[
    \mb S/I^{\wedge n} \wedge X \to \mb S/I^{\wedge (n-1)} \wedge X
    \to E \wedge I^{\wedge (n-1)} \wedge X.
  \]
  By induction on $n$ we find that $\mb S/I^{\wedge n} \wedge X$
  is $E$-local, and so the homotopy limit $X^\wedge_E$ is $E$-local.

  After smashing with $E$, the cofiber sequence
  \[
    E \wedge I^{\wedge n} \wedge X \to E \wedge I^{\wedge (n-1)}
    \wedge X \to E \wedge E \wedge
    I^{\wedge (n-1)} \wedge X
  \]
  has a retraction of the second map via the (opposite) multiplication
  of $E$, and so the first map is nullhomotopic. Therefore, the
  homotopy limit $\holim E \wedge (I^{\wedge n} \wedge X)$ is trivial, and
  from the cofiber sequences
  \[
    E \wedge (I^{\wedge n} \wedge X) \to E \wedge X \to E \wedge (\mb
    S/I^{\wedge n} \wedge X)
  \]
  we find that $E \wedge X \to \holim(E \wedge (\mb S/I^{\wedge n} \wedge
  X)$ is an equivalence.

  This reduces us to proving that the map
  \[
    E \wedge \holim(\mb S/I^{\wedge n} \wedge X) \to
    \holim(E \wedge \mb S/I^{\wedge n} \wedge X)
  \]
  is an equivalence: we can move the smash product with $E$ inside the
  homotopy limit.  This is always true if $E$ is finite or if $E$ is
  of finite type and the homotopy limit is of connective objects.
\end{proof}

\begin{rmk}
  The spectral sequence arising from the inverse system defining
  $X^\wedge_E$ is the \emph{generalized Adams--Novikov spectral
    sequence based on $E$-homology}. It often abuts to the homotopy
  groups of the Bousfield localization with respect to $E$.

  We can generalize our construction by allowing more general towers
  with a nilpotence property, after Bousfield in 
  \cite{bousfield-spectralocalization}, or by extending these methods
  to the category of modules over a ring spectrum, as Baker--Lazarev
  did in \cite{baker-lazarev-adamss} or Carlsson did in
  \cite{carlsson-completions}.
\end{rmk}

\begin{exam}
  For any prime $p$ and any $n > 0$, we have the Johnson--Wilson
  homology theories $E(n)_*$ and the Morava $K$-theories
  $K(n)_*$. Associated to these we have $E(n)$-localization functors
  and $K(n)$-localization functors, as well as categories of
  $E(n)$-local and $K(n)$-local spectra, which play an essential role
  in chomatic homotopy theory. Ravenel conjectured, and
  Devinatz--Hopkins--Smith proved, that the localization $L_{E(n)}$ is
  a smashing localization \cite{ravenel-localization,
    devinatz-hopkins-smith-nilpotence, ravenel-orangebook}. These
  localizations also have \emph{chromatic fractures} which are built
  using the following result.
\end{exam}

\begin{prop}
  Suppose that $E$ and $K$ are spectra such that $L_K L_E X$ is always
  trivial. Then, for all $X$, there is a homotopy pullback
  diagram
  \[
    \xymatrix{
      L_{E \vee K} X \ar[r] \ar[d] & L_E X \ar[d]\\
      L_K X \ar[r] & L_E L_K X.
    }
  \]
\end{prop}

\begin{proof}
  The objects in the diagram
  \[
    L_E X \to L_E L_K X \leftarrow L_K X
  \]
  are either $E$-local or $K$-local, and hence automatically
  $E \vee K$-local; therefore, the homotopy pullback $P$ is $E
  \vee K$-local. It then suffices to show that the fiber of the map
  $X \to P$ is $E \vee K$-trivial, which is equivalent to showing
  that
  \[
    \xymatrix{
      X \ar[r] \ar[d] & L_E X \ar[d]\\
      L_K X \ar[r] & L_E L_K X.
    }
  \]
  becomes a homotopy pullback after smashing with $E \vee K$.
  After smashing with $E$, the horizontal maps become equivalences,
  and so the diagram is a pullback. After smashing with $K$, the
  left-hand vertical map is an equivalence and the right-hand vertical
  map is between trivial objects, so the diagram is also a
  pullback. Therefore, the diagram becomes a pullback after smashing
  with $E \vee K$.
\end{proof}

\section{Model categories}
\label{sec:model-categories}

The lifting characterization of local objects from
\S\ref{sec:lifting} falls very naturally into the framework of
Quillen's model categories. The groundwork for this is in \cite[\S
10]{bousfield-spacelocalization}.

\begin{defn}
  Suppose that $\mathcal{M}$ is a category with a model structure. We
  say that a second model structure $\mathcal{M}'$ with the same
  underlying category is a \emph{left Bousfield localization} of
  $\mathcal{M}$ if $\mathcal{M}'$ has the same family of cofibrations
  but a larger family of weak equivalences than $\mathcal{M}$.
\end{defn}

As a first consequence, note that the identity functor (which is its
own right and left adjoint) preserves cofibrations and takes the weak
equivalences in $\mathcal{M}$ to weak equivalences in
$\mathcal{M}'$. This makes it part of a Quillen adjunction
\[
  \mathcal{M} \rightleftarrows \mathcal{M}'.
\]
This has the immediate consequence that the induced adjunction on
homotopy categories is a reflective localization.
\begin{prop}
  Suppose that
  $L\co \mathcal{M} \rightleftarrows \mathcal{M}' \co R$
  is the adjunction associated to a left Bousfield localization. Then
  the right adjoint $R$ identifies the homotopy category $h\mathcal{M}'$
  with a full subcategory of $h\mathcal{M}$.
\end{prop}

\begin{proof}
  It is necessary and sufficient to show that the counit
  $\epsilon\co LR x \to x$ of the adjunction on homotopy categories is
  always an isomorphism, for this is the same as asking that, in the
  factorization
  \[
    \Hom_{h\mathcal{M}}(Rx,Ry) \cong
    \Hom_{h\mathcal{M}'}(LRx,y) \to
    \Hom_{h\mathcal{M}'}(x,y),
  \]
  the second map is an isomorphism.
  
  For an object of $y$, the composite functor $LR$ on homotopy
  categories is calculated as follows: find a fibrant replacement
  $y \too{\simeq'} y_{f'}$ in $\mathcal{M}'$, apply the identity
  functor to get to $\mathcal{M}$, find a cofibrant replacement
  $(y_{f'})_c \too{\simeq} y_{f'}$ in $\mathcal{M}$, and apply the
  identity functor to get to $\mathcal{M}'$. The counit of the
  adunction is represented in the homotopy category of $\mathcal{M}'$
  by the composite
  \[
    (y_{f'})_c \too{\simeq} y_{f'} \xleftarrow{\simeq'} y.
  \]
  However, equivalences in $\mathcal{M}$ are automatically
  equivalences in $\mathcal{M}'$, and so the counit is an
  isomorphism in the homotopy category of $\mathcal{M}'$.
\end{proof}

Because fibrations and acyclic fibrations are determined by having the
right lifting property against acyclic cofibrations and fibrations,
the new model structure has the same acyclic fibrations but fewer
fibrations. For example, a fibrant object in the left Bousfield
localization has to have a lifting property against the cofibrations
which are weak equivalences in $\mathcal{M}'$.

The next proposition establishes the connection between left Bousfield
localization and ordinary Bousfield localization when both are defined
and compatible: the case of a simplicial model category.

\begin{prop}
  Suppose that $\mathcal{M}$ is a simplicially enriched category with
  two model structures, making $\mathcal{M} \to \mathcal{M}'$ is a
  left Bousfield localization of simplicial model categories. Let $S$
  be the collection of weak equivalences between cofibrant objects in
  $\mathcal{M}'$. Then, in the category of cofibrant-fibrant objects
  of $\mathcal{M}$, the objects which are fibrant in $\mathcal{M}'$
  are precisely the $S$-local fibrant objects.
\end{prop}

\begin{proof}
  Fix an object $Y$ of $\mathcal{M}'$. For it to be fibrant in
  $\mathcal{M}'$, it must also be fibrant in $\mathcal{M}$.
  Suppose $Y$ is a fibrant object in $\mathcal{M}'$. Given any acyclic
  cofibration $A \to B$ in $\mathcal{M}'$, the map of simplicial sets
  $\Map_{\mathcal{M'}}(A,Y) \to \Map_{\mathcal{M}'}(B,Y)$ is an
  acyclic fibration by the SM7 axiom of simplicial model
  categories. Thus, the functor $\Map_{\mathcal{M}'}(-,Y)$ from
  $\mathcal{M}'$ to the homotopy category of spaces takes acyclic
  cofibrations to isomorphisms. Thus, Ken Brown's lemma implies that
  it also takes weak equivalences between cofibrant objects in
  $\mathcal{M}'$ to isomorphisms in the homotopy category of spaces.

  Suppose that we have a map $f\co A \to B$ in $S$ between cofibrant
  objects of $\mathcal{M}$ that is also a weak equivalence in
  $\mathcal{M}'$. Then $f$ is also a weak equivalence between
  cofibrant objects of $\mathcal{M}'$. The induced map
  $\Map_{\mathcal{M}}(B,Y) \to \Map_{\mathcal{M}}(A,Y)$ is a weak
  equivalence because the mapping spaces in $\mathcal{M}$ and
  $\mathcal{M}'$ are the same. Thus, $Y$ is $S$-local.
\end{proof}

We would now like to establish results in the other direction. Namely,
given a model category $\mathcal{M}$ and a collection $S$ of maps
$A_i \to B_i$ in $\mathcal{M}$, we would like to establish the
existence of a Bousfield localization $\mathcal{M}'$ of
$\mathcal{M}$. Because we want to work within the already-established
homotopy theory of $\mathcal{M}$, we want to use derived mapping
spaces out of $A$ and $B$ and replace homotopy lifting
properties with strict lifting properties. We assume without loss of
generality that our set $S$ is made up of cofibrations $A_i \to B_i$
between cofibrant objects.

\begin{defn}
  Suppose that $\mathcal{M}$ is a simplicial model category, and that
  $f\co A \to B$ is a map. Then the \emph{iterated double mapping
    cylinders} are the maps
  \[
    (B \otimes \partial \Delta^n) \coprod_{A \otimes \partial \Delta^n}
    (A \otimes \Delta^n) \to B \otimes \Delta^n.
  \]
\end{defn}

This definition is rigged so that an object $Y$ has the right lifting
property with respect to the iterated double mapping cylinders if and
only if the map $\Map_{\mathcal{M}}(B,Y) \to \Map_{\mathcal{M}}(A,Y)$
is an acyclic fibration of simplicial sets. One of the
equivalent formulations of the SM7 axioms for a simplicial model
category is that double mapping cylinders are always cofibrations, as
follows.
\begin{prop}
  Suppose that $f\co A \to B$ is a map. If $f$ is a cofibration, then
  the iterated double mapping cylinders are cofibrations. If  $A$ is
  also cofibrant, then the iterated double mapping cylinders have
  cofibrant source.
\end{prop}

\begin{rmk}
  If $\mathcal{M}$ does not have a simplicial model structure, we can
  obtain replacements for these objects by iteratively replacing the
  maps $B \coprod_A B \to B$ with equivalent cofibrations.
\end{rmk}

\begin{defn}
  Suppose that $\mathcal{M}$ is a simplicial model category, that $S$
  is a collection of maps, and that $T$ is the collection of iterated
  double mapping cylinders of maps in $S$. We say that a map in
  $\mathcal{M}$ is an \emph{$S$-cofibration} if it is a cofibration in
  $\mathcal{M}$, and that it is an \emph{$S$-fibration} if it has the
  right lifting property with respect to the maps in $T$. If these
  determine a new model structure $\mathcal{M}'$, we call this the
  \emph{left Bousfield localization with respect to $S$}.
\end{defn}

This gives us two fundamentally different approaches to the process of
constructing a left Bousfield localization. In the first, we may try
to expand our family of weak equivalences to some new family
$\mathcal{W}$; we must then prove that we can construct enough
fibrations and fibrant objects to make the model structure work. In
the second, we may try to start with some collection of maps $S$ which
serve as new ``cells'' to build acyclic cofibrations, and use
them to contract our family of fibrations; we then lose control over
the weak equivalences, and typically must work to prove that
cofibrations which are weak equivalences can be built out of our new
cells.

The most advanced technology available for Bousfield localization is
Jeff Smith's theory of combinatorial model categories.

\begin{defn}
  \label{def:cofibrantlygenerated}
  A model category $\mathcal{M}$ is \emph{cofibrantly generated} if
  there are sets $I$ and $J$ of maps satisfying the following properties:
  \begin{enumerate}
  \item the fibrations in $\mathcal{M}$ are the maps that have the
    right lifting property with respect to $J$;
  \item the acyclic fibrations in $\mathcal{M}$ are the maps that have
    the right lifting property with respect to $I$;
  \item $I$ permits the small object argument, so that from any
    object $X$ we can construct a map $X \to X'$, as a transfinite
    composition of pushouts along coproducts of maps in $I$, that has
    the right lifting property with respect to $I$;
  \item $J$ also permits the small object argument.
  \end{enumerate}
  We refer to $I$ as the set of \emph{generating cofibrations} and to
  $J$ as the set of \emph{generating acyclic cofibrations}
  respectively.
  
  The cofibrantly generated model category is also
  \emph{combinatorial} if it is also locally presentable, meaning
  there exists a regular cardinal $\kappa$ and a set $\mathcal{M}_0$
  of objects satisfying the following properties:
  \begin{enumerate}
  \item any small diagram in $\mathcal{M}$ has a colimit;
  \item for any object $x$ in $\mathcal{M}_0$, the functor
    $\Hom_{\mathcal{M}}(x,-)$ commutes with $\kappa$-filtered
    colimits;
  \item every object in $\mathcal{M}$ is a $\kappa$-filtered colimit
    of objects in $\mathcal{M}_0$.
  \end{enumerate}
\end{defn}

\begin{thm}[Dugger's theorem {\cite{dugger-universal}}]
  Any combinatorial model category is Quillen equivalent to a left
  proper simplicial model category.
\end{thm}

\begin{rmk}
  The axioms of a cofibrantly generated model category and a locally
  presentable category have nontrivial overlap. In one direction, the
  model category axioms already ask that $\mathcal{M}$ has all
  colimits. In the other direction, being locally presentable means
  that \emph{every} set of maps admits the small object argument.
\end{rmk}

\begin{exam}
  Simplicial sets are the motivating example of a combinatorial model
  category. Fibrations and acyclic fibrations are defined as having
  the right lifting property with respect to the generating acyclic
  cofibrations $\Lambda^n_i \to \Delta^n$ and the generating
  cofibrations $\partial \Delta^n \to \Delta^n$. The category is also
  locally presentable because it is generated by finite simplicial
  sets. Every simplicial set is the filtered colimit of its finite
  subobjects; there are only countably many isomorphism classes of
  finite simplicial sets; for any finite simplicial set $X$,
  $\Hom(X,-)$ commutes with filtered colimits.
\end{exam}

\begin{thm} [Smith's theorem
  {\cite{beke-sheafifiable,barwick-bousfieldlocalization,lurie-htt}}]
  Suppose that $\mathcal{M}$ is a locally presentable category with a
  family $\mathcal{W}$ of \emph{weak equivalences} and a set $I$ of
  \emph{generating cofibrations}. Call those maps which have the right
  lifting property with respect to $I$ the \emph{acyclic fibrations},
  and those maps which have the left lifting property with respect to
  acyclic fibrations the \emph{cofibrations}. Suppose that we have the
  following:
  \begin{enumerate}
  \item $\mathcal{W}$ satisfies the 2-out-of-3 axiom;
  \item acyclic fibrations are in $\mathcal{W}$;
  \item the class of cofibrations which are in $\mathcal{W}$ is closed
    under pushout and transfinite composition; and
  \item maps in $\mathcal{W}$ are closed under $\kappa$-filtered
    colimits for some regular cardinal $\kappa$, and generated under
    $\kappa$-filtered colimits by some set of maps in $\mathcal{W}$.
  \end{enumerate}
  Then there exists a combinatorial model structure on $\mathcal{M}$
  with set $I$ of generating cofibrations and set $\mathcal{W}$ of
  weak equivalences. This model structure on $\mathcal{M}$ has
  cofibrant and fibrant replacement functors. Moreover, any
  combinatorial model structure arises in this fashion.
\end{thm}

\begin{cor}
  Suppose that $\mathcal{M}$ is a combinatorial model category with
  set $I$ of generating cofibrations and class $\mathcal{W}$ of weak
  equivalences. Given a functor $E\co \mathcal{M} \to \mathcal{D}$
  factoring through the homotopy category $h\mathcal{M}$, define a map
  to be an \emph{$E$-equivalence} if its image under $E$ is an
  isomorphism. Then there exists a left Bousfield localization
  $\mathcal{M}_E$, whose equivalences are the $E$-equivalences, if the
  following conditions hold:
  \begin{enumerate}
  \item $E$-equivalence is preserved by transfinite composition along
    cofibrations;
  \item pushouts of $E$-acyclic cofibrations are $E$-equivalences; and
  \item there exists a set of $E$-acyclic cofibrations that generate
    all $E$-acyclic cofibrations under $\kappa$-filtered colimits.
  \end{enumerate}
\end{cor}

\begin{proof}
  The 2-out-of-3 axiom is automatic: if two of $E(g)$, $E(f)$ and
  $E(gf) = E(g) E(f)$ are isomorphisms, then so is the third. The
  fact that $E$ factors through the homotopy category automatically
  implies that acyclic fibrations are taken by $E$ to isomorphisms.
\end{proof}

\begin{exam}
  Let $E_*$ be a homology theory on the category of simplicial
  sets. The excision and direct limit axioms for homology imply that
  $E$-equivalences are preserved by homotopy pushouts and transfinite
  compositions. Therefore, the verification that we have a model
  structure is immediately reduced to the core of the Bousfield--Smith
  cardinality argument of Example~\ref{rmk:bousfield-smith}: that
  there is a set of $E$-acyclic cofibrations generating all others
  under filtered colimits.
\end{exam}

The great utility of combinatorial model structures is that they allow
us to \emph{build} new model categories: categories of diagrams and
Bousfield localizations.

\begin{thm}[{\cite[A.2.8.2, A.3.3.2]{lurie-htt}}]
  Suppose that $\mathcal{M}$ is a combinatorial model category and
  that $I$ is a small category. Then there exists a \emph{projective}
  (resp. \emph{injective}) model structure on the functor category
  $\mathcal{M}^I$, where a natural transformation of diagrams is an
  equivalence or fibration (resp. cofibration) if and only if it is an
  objectwise equivalence or fibration (resp. cofibration).

  If $\mathcal{M}$ is a simplicial model category, then the natural
  simplicial enrichment on $\mathcal{M}^I$ makes the injective and
  projective model structures into simplicial model categories.
\end{thm}

\begin{thm}[{\cite[A.3.7.3]{lurie-htt}}]
  Suppose that $\mathcal{M}$ is a left proper combinatorial simplicial
  model category and that $S$ is a set of cofibrations in
  $\mathcal{M}$. Let $S^{-1} \mathcal{M}$ have the same underlying
  category as $\mathcal{M}$ and the same cofibrations, but with weak
  equivalences the $S$-equivalences.

  Then $S^{-1} \mathcal{M}$ has the structure of a left proper
  combinatorial model category, whose fibrant objects are precisely
  the $S$-local fibrant objects of $\mathcal{M}$.
\end{thm}

%


\section{Presentable $\infty$-categories}
\label{sec:infty-categories}

Bousfield localization for model categories has the useful property
that it \emph{keeps the category in place} and merely changes the
equivalences. One cost is that making localization canonical or
extending monoidal structures to localized objects takes hard work. By
contrast, localization for $\infty$-categories has the useful property
that it is genuinely \emph{defined by a universal property},
automatically making localization canonical and making it much easier
to extend a monoidal structure to local objects without rectifying
structure. Of course, this comes at the cost of coming to grips with
coherent category theory itself.

The homotopy theory of presentable $\infty$-categories is equivalent,
in a precise sense, to the homotopy theory of combinatorial model
categories \cite[A.3.7.6]{lurie-htt}. However, by contrast with our
techniques for Bousfield localization using model categories and
fibrant replacement functors, it allows us to rephrase some of our
localization techniques in a way that connects more directly with the
homotopical techniques that we originally used in
\S\ref{sec:lifting}.

In this section, we will let $\mathcal{C}$ be an $\infty$-category in
the sense of \cite{lurie-htt}. It is outside our scope to give a
technically correct discussion of these. However, the study of
$\infty$-categories is equivalent to the study of categories with
morphism spaces, and where possible we will attempt to make connection
with classical techniques. With this in mind, if $\mathcal{C}$ is an
enriched category we will say that a \emph{coherent diagram}
$I \to \mathcal{C}$ is a coherent functor in the sense of Vogt
\cite{vogt-hocolim}. This is equivalent to either the notion of a
functor $\mf C[I] \to \mathcal{C}$ from a certain simplicially
enriched category or to the notion of a functor $I \to N\mathcal{C}$
of simplicial sets to the coherent nerve in the sense of
\cite{lurie-htt}. As before a \emph{homotopy colimit} for such a
diagram is based on classical homotopy limits and colimits in spaces,
and is characterized by having natural weak equivalences
\[
  \Map_{\mathcal{C}}(\hocolim_I F(i), Y) \simeq \holim_I
  \Map_{\mathcal{C}}(F(i),Y).
\]

\begin{defn}[{\cite[5.5.1.1]{lurie-htt}}]
  An $\infty$-category $\mathcal{C}$ is \emph{presentable} if there
  there exists a regular cardinal $\kappa$ and a set $\mathcal{C}_0$
  of objects satisfying the following properties:
  \begin{enumerate}
  \item any small diagram in $\mathcal{C}$ has a homotopy colimit;
  \item for any object $x$ in $\mathcal{C}_0$, the functor
    $\Hom_{\mathcal{C}}(x,-)$ commutes with $\kappa$-filtered
    homotopy colimits;
  \item every object in $\mathcal{C}$ is a $\kappa$-filtered homotopy
    colimit of objects in $\mathcal{C}_0$.
  \end{enumerate}
\end{defn}
This definition is precisely parallel to the definition of local
presentability in an ordinary category (see
Definition~\ref{def:cofibrantlygenerated}). In essence, $\mathcal{C}$
is a large category that is formally generated under colimits by a
small category.

Given such an $\infty$-category $\mathcal{C}$ and a collection $S$ of
morphisms in $\mathcal{C}$, it makes sense to define the $S$-local
objects and $S$-equivalences just as in \S\ref{sec:mappingspaces}: an
object $Y$ is $S$-local if and only if the mapping spaces
$\Map_{\mathcal{C}}(-,Y)$ take maps in $S$ to equivalences of spaces.

\begin{defn}[{\cite[5.5.4.5]{lurie-htt}}]
  Suppose that $\mathcal{C}$ is an $\infty$-category with small
  colimits and that $\mathcal{W}$ is a collection of maps in
  $\mathcal{C}$. We say that $\mathcal{W}$ is \emph{strongly
    saturated} if it satisfies the following conditions:
  \begin{enumerate}
  \item given a homotopy pushout diagram
    \[
      \xymatrix{
        C \ar[r]^f \ar[d] & D \ar[d] \\
        C' \ar[r]_{f'} & D',
      }
    \]
    if $f$ is in $\mathcal{W}$ then so is $f'$;
  \item the class $\mathcal{W}$ is closed under homotopy colimits;
  \item the class $\mathcal{W}$ is closed under equivalence, and its
    image in the homotopy category satisfies the 2-out-of-3 axiom.
  \end{enumerate}
\end{defn}

\begin{prop}[{\cite[5.5.4.7]{lurie-htt}}]
  Given a set $S$ of morphisms in $\mathcal{C}$, there is a
  smallest saturated class of morphisms containing $S$. We denote this
  as $\bar S$. If $\mathcal{W} = \bar S$ for some set $S$, then we say
  that $\mathcal{W}$ is \emph{of small generation.}
\end{prop}

\begin{exam}
  Suppose that $E\co \mathcal{C} \to \mathcal{C}'$ is a functor of
  $\infty$-categories that preserves homotopy colimits. Then the set
  $\mathcal{W}^E$ of maps in $\mathcal{C}$ that map to equivalences
  is strongly saturated.
\end{exam}

The presentability axioms for an $\infty$-category provide a
homotopical version of what we needed to construct localizations by
ensuring that the small object argument goes through. As a result, we
obtain a result on the existence of Bousfield localizations for
presentable $\infty$-categories.
\begin{thm}[{\cite[5.5.4.15]{lurie-htt}}]
  Let $\mathcal{C}$ be a presentable $\infty$-category and $S$ a set
  of morphisms in $\mathcal{C}$, generating the saturated class $\bar
  S$. Let $L^S\mathcal{C}$ be the full subcategory of $S$-local
  objects. Then the following hold:
  \begin{enumerate}
  \item for every object $C \in \mathcal{C}$, there is a map $C \to
    C'$ in $\bar S$ such that $C'$ is $S$-local;
  \item the $\infty$-category $L^S\mathcal{C}$ is presentable;
  \item the inclusion $L^S\mathcal{C} \to \mathcal{C}$ has a (homotopical)
    left adjoint $L$;
  \item the class of $S$-equivalences coincides with both the
    saturated class $\bar S$ and the set of maps taken to equivalences
    by $L$.
  \end{enumerate}
\end{thm}

\begin{rmk}
  The homotopical left adjoint can be rephrased as follows. If we
  write $Loc^S(\mathcal{C})$ for the category of $S$-localizations $C
  \to C'$, then the forgetful functor
  \[
    Loc^S(\mathcal{C}) \to \mathcal{C},
  \]
  sending $(C \to C')$ to $C$, is an equivalence of categories (in
  fact, a trivial fibration of quasicategories). By choosing a
  section, given by $C \mapsto (C \to LC)$, we obtain a localization
  functor $L$.
\end{rmk}

As in the case of Bousfield localization of combinatorial model
categories, this connects the two approaches to Bousfield
localization. We can start with a set $S$ of generating equivalences
and construct localizations from those, so for a given class
$\mathcal{W}$ of weak equivalences we are reduced to showing that
$\mathcal{W}$ is generated by a set $S$ of maps. Moreover, if the maps
in $S$ all happen to be in a particular saturated class, then so are
the maps in $\mathcal{W}$.

\section{Multiplicative properties}

Many of the categories where we carry out Bousfield localization have
monoidal structures, and under good circumstances localization is
compatible with them. In this section we will briefly discuss the
circumstances under which this is true.

\subsection{Enriched monoidal structures}

In order to begin to work with these definitions, we need a monoidal
or symmetric monoidal structure on $\mathcal{C}$ that respects
morphism spaces.

\begin{defn}
  Suppose $\mathcal{C}$ is a category enriched in spaces. The
  structure of an \emph{enriched monoidal category} on $\mathcal{C}$
  consists of a functor
  $\otimes\co \mathcal{C} \times \mathcal{C} \to \mathcal{C}$ of
  enriched categories, a unit object $\mb I$ of $\mathcal{C}$,
  and natural associativity and commutativity isomorphisms that
  satisfy the axioms for a monoidal category.

  A compatible symmetric monoidal structure on $\mathcal{C}$ is
  defined similarly.
\end{defn}

Throughout this section we will fix such an enriched monoidal
category $\mathcal{C}$.

\begin{defn}
  Suppose that $S$ is a class of morphisms in $\mathcal{C}$. We say
  that $S$-equivalences are \emph{compatible with the monoidal
    structure} (or simply that $S$ is compatible) if, for any
  $S$-equivalence $f\co Y \to Y'$ and any object $X \in \mathcal{C}$,
  the maps $id_X \otimes f$ and $f \otimes id_X$ are $S$-equivalences.
\end{defn}

\begin{prop}
  Suppose that $S$ is compatible with the monoidal structure. Then
  localization respects the monoidal structure: any choices of
  localization give an equivalence
  \[
    L(X_1 \otimes \dots \otimes X_n) \to
    L(LX_1 \otimes \dots \otimes LX_n).
  \]
\end{prop}

\begin{proof}
  By induction, the map $X_1 \otimes \dots \otimes X_n \to LX_1
  \otimes \dots \otimes LX_n$ is an $S$-equivalence, and therefore any
  $S$-localization of the latter is equivalent to any $S$-localization
  of the former.
\end{proof}

\begin{cor}
  The monoidal structure on the homotopy category of $\mathcal{C}$
  induces a monoidal structure on the homotopy catogory of the
  localization $L^S \mathcal{C}$, making any localization functor into
  a monoidal functor. If $\mathcal{C}$ was symmetric monoidal, then so
  is the localization.
\end{cor}

\begin{rmk}
  The inclusion $L^S \mathcal{C} \to \mathcal{C}$ is almost never
  monoidal. For example, it usually does not preserve the unit.
\end{rmk}

\begin{exam}
  Let $\mathcal{C}$ be the category of spaces with cartesian product,
  and let $E_*$ be a homology theory. Then any map $X \to X'$ which
  induces an isomorphism on $E_*$-homology also induces
  isomorphisms $E_*(X \times Y) \to E_*(X' \times Y)$ for any
  CW-complex $Y$: one can prove this inductively on the cells
    of $Y$. Therefore, $E$-homology equivalences are compatible with the
  Cartesian product monoidal structure.

  Similarly, $E$-homology equivalences are compatible with the smash
  product on based spaces (using that based spaces are built from
  $S^0$) or the smash product on spectra (using that all spectra are
  built from spheres $S^n$).
\end{exam}

\begin{exam}
  Let $\mathcal{C}$ be the category of spectra, and $f$ be the map
  $S^n \to *$. Then $f$-equivalences are maps inducing isomorphisms in
  degree strictly less than $n$. This is not compatible with the smash
  product on spectra: for example, smashing with $\Sigma^{-1} \mb S$
  does not preserve $f$-equivalences. If one restricts to the
  subcategory of \emph{connective} spectra, however, one finds that
  $f$-equivalences are compatible with the smash product.
\end{exam}

\begin{exam}
  Consider the map $f\co S^n \to *$ of spaces, so that
  $S$-equivalences are maps inducing an isomorphism on all homotopy
  groups in degrees less than $n$. This map is compatible with several
  symmetric monoidal structures, such as:
  \begin{enumerate}
  \item spaces with Cartesian product;
  \item spaces with disjoint union;
  \item based spaces with wedge product; and
  \item based spaces with smash product.
  \end{enumerate}
\end{exam}

Despite the usefulness of these results, the existence of a
(symmetric) monoidal localization functor on the homotopy category
does not, by itself, allow us to extend very structured multiplication
from an object $X$ to its localization $LX$. To counter this we
typically require the theory of operads.

\begin{defn}
  Suppose that $\mathcal{C}$ is (symmetric) monoidal, and that $X$ is
  an object of $\mathcal{C}$. The endomorphism operad
  $\End_{\mathcal{C}}(X)$ is the (symmetric) sequence of spaces
  $\Map_{\mathcal{C}}(X \otimes \dots \otimes X, X)$, with (symmetric)
  operad structure given by composition.

  Given a map $f\co X \to Y$, the endomorphism operad
  $\End_{\mathcal{C}}(f)$ is the (symmetric) sequence which in degree
  $n$ is the pullback diagram
  \[
    \xymatrix{
      \End_{\mathcal{C}}(f)_n \ar[r]  \ar[d] &
      \Map_{\mathcal{C}}(X \otimes \dots \otimes X, X) \ar[d] \\
      \Map_{\mathcal{C}}(Y \otimes \dots \otimes Y, Y) \ar[r] &
      \Map_{\mathcal{C}}(X \otimes \dots \otimes X, Y).
    }      
  \]
  The space $\End_{\mathcal{C}}(f)_n$ is the space of strictly
  commutative diagrams
  \[
    \xymatrix{
      X^{\otimes n} \ar[r] \ar[d]_{f^{\otimes n}} & X \ar[d]^f \\
      Y^{\otimes n} \ar[r]& Y,
    }
  \]
  and as such the operad structure is given by composition.
\end{defn}

The operad $\End_{\mathcal{C}}(f)$ has forgetful maps to
$\End_{\mathcal{C}}(X)$ and $\End_{\mathcal{C}}(Y)$.

\begin{prop}
  Suppose that the (symmetric) monoidal structure on $\mathcal{C}$ is
  compatible with $S$ and that $f\co X \to LX$ is an
  $S$-localization. If the maps $\Map_{\mathcal{C}}(LX^{\otimes n}, LX)
  \to \Map_{\mathcal{C}}(X^{\otimes n}, LX)$ are fibrations for all $n
  \geq 0$, then in the diagram of operads
  \[
    \End_{\mathcal{C}}(X) \leftarrow \End_{\mathcal{C}}(f)
    \to \End_{\mathcal{C}}(LX),
  \]
  the left-hand arrow is an equivalence on the level of underlying
  spaces.
\end{prop}

\begin{proof}
  This is merely the observation that $\End_{\mathcal{C}}(f)
  \to \End_{\mathcal{C}}(X)$ is, level by level, a homotopy pullback
  of the equivalences $\Map_{\mathcal{C}}(LX^{\otimes n}, LX) \to
  \Map_{\mathcal{C}}(X^{\otimes n}, LX)$.
\end{proof}

This condition then allows us to lift structured multiplication.
\begin{cor}
  Suppose that a (symmetric) operad $\mathcal{O}$ acts on $X$ via a map
$\mathcal{C} \to \End_{\mathcal{C}}(X)$. Then there exists a weak
equivalence $\mathcal{O}' \to \mathcal{O}$ of operads and an action of
$\mathcal{O}'$ on $LX$ such that $f$ is a map of
$\mathcal{O}'$-algebras.
\end{cor}

\begin{proof}
  We define $\mathcal{O}'$ to be the fiber product of the diagram
  $\mathcal{O} \to \End_{\mathcal{C}}(X)
  \leftarrow \End_{\mathcal{C}}(f)$. The map
  $\mathcal{O}' \to \mathcal{O}$ is an equivalence by the fibration
  condition, and the map $\mathcal{O}' \to \End_{\mathcal{C}}(f)$ of
  operads precisely states that $f$ is a map of
  $\mathcal{O}'$-algebras. \footnote{If $\mathcal{O}$ happens to be a
    cofibrant (symmetric) operad $\mathcal{O}$ in Berger--Moerdijk's
    model structure \cite{berger-moerdijk-enrichedcats} we can do
    better. Any map $\mathcal{O} \to \End_{\mathcal{C}}(X)$ lifts, up
    to homotopy, to a map
    $\mathcal{O} \to \End_{\mathcal{C}}(f)
    \to \End_{\mathcal{C}}(LX)$.}
\end{proof}
This means that $A_\infty$ and $E_\infty$ multiplications on $X$
extend automatically to $A_\infty$ and $E_\infty$ multiplications on
$LX$.  However, this is the best we can do in general: lifting more
refined multiplicative structures requires stronger assumptions.

In cases where the category $\mathcal{C}$ has more structure, it is
typically easier to verify that $S$ is compatible with the monoidal
structure.

\begin{prop}
  Suppose that the monoidal structure on $\mathcal{C}$ has internal
  function objects $F^L(X,Y)$ and $F^R(X,Y)$ that are adjoint to
  the monoidal structure: there are isomorphisms
  \[
    \Map_{\mathcal{C}}(X,F^L(Y,Z)) \cong \Map(X \otimes Y, Z) \cong
    \Map_{\mathcal{C}}(Y,F^R(X,Z))
  \]
  that are natural in $X$, $Y$, and $Z$.  Then $S$ is compatible with
  the monoidal structure on $\mathcal{C}$ if and only if, for any
  $f\co A \to B$ in $S$ and any object $X \in \mathcal{C}$, the maps
  $id_X \otimes f$ and $f \otimes id_X$ are $S$-equivalences.
\end{prop}

\begin{proof}
  Suppose that for any $f\co A \to B$ in $S$ and any object
  $X \in \mathcal{C}$, the maps $id_X \otimes f$ are
  $S$-equivalences. Using the unit isomorphisms, we find that if $Z$
  is $S$-local the maps in the diagram
  \[
    \xymatrix{
      \Map_{\mathcal{C}}(X \otimes B, Z) \ar[r] \ar[d] &
      \Map_{\mathcal{C}}(X \otimes A, Z) \ar[d] \\
      \Map_{\mathcal{C}}(B,F^R(X,Z)) \ar[r] &
      Map_{\mathcal{C}}(A,F^R(X,Z))
    }
  \]
  are equivalences. Therefore, $F^R(X,Z)$ is $S$-local, and so for any
  $S$-equivalence $f\co Y \to Y'$ the maps in the diagram
  \[
    \xymatrix{
      \Map_{\mathcal{C}}(X \otimes Y', Z) \ar[r] \ar[d] &
      \Map_{\mathcal{C}}(X \otimes Y, Z) \ar[d] \\
      \Map_{\mathcal{C}}(Y',F^R(X,Z)) \ar[r] &
      Map_{\mathcal{C}}(Y,F^R(X,Z))
    }
  \]
  are all equivalences. Similar considerations apply to $F^L$.
\end{proof}

\subsection{Monoidal model categories}

The necessary conditions for compatibility between model structures
and monoidal structures were determined by Schwede--Shipley
\cite{shipley-schwede-algebrasmodules} and Hovey \cite[\S
4.2]{hovey-modelcategories}, in the symmetric and nonsymmetric cases
respectively. This structure allows us, after
\cite{shipley-schwede-algebrasmodules}, to construct model structures
on categories of algebras and modules in $\mathcal{M}'$ such that
the localization functor $\mathcal{M} \to \mathcal{M}'$ preserves
this structure.

\begin{defn}
  A \emph{(symmetric) monoidal model category} $\mathcal{M}$ is a
  model category with a (symmetric) monoidal closed
  structure\footnote{Analogously to the previous section, this means
    that the symmetric monoidal structure must have left and right
    function objects which are adjoints in each variable.} satisfying
  the following axioms.
  \begin{enumerate}
  \item (Pushout-product) Given cofibrations $i\co A \to A$ and $j\co
    B \to B'$ in $\mathcal{M}$, the induced pushout-product map
    \[
      i \boxtimes j\co (A \otimes B') \coprod_{A \otimes B} (A' \otimes B) \to
      A' \otimes B'
    \]
    is a cofibration, which is acyclic if either $i$ or $j$ is.
  \item (Unit) Let $Q\mb I \to \mb I$ be a cofibrant replacement of the
    unit. Then the natural maps $Q\mb I \otimes X \to X \leftarrow X
    \otimes Q\mb I$ are isomorphisms for all cofibrant $X$.
  \end{enumerate}
\end{defn}

\begin{prop}
  \label{prop:tensorwithcofibrant}
  Suppose that $\mathcal{M}$ is a monoidal model category. Then, for
  cofibrant objects $X$, the functors $X \otimes (-)$ and $(-) \otimes
  X$ preserve cofibrations, acyclic cofibrations, and weak
  equivalences between cofibrant objects.
\end{prop}

\begin{proof}
  Since $\otimes$ has adjoints, it preserves colimits in each
  variable. In particular, any object tensored with an initial object
  of $\mathcal{M}$ is an initial object of $\mathcal{M}$. Applying the
  pushout-product axiom to the map $\emptyset \to X$ in either
  variable, we find that the two functors in question preserve
  cofibrations and acyclic cofibrations. By Ken Brown's lemma, they
  also automatically take weak equivalences between cofibrant objects
  to weak equivalences.
\end{proof}

This connects with our work in the the previous section, which only
asked that the tensor product preserved equivalences in each variable.
The pushout-product axiom for monoidal model categories looks
stronger, in principle, but Proposition~\ref{prop:tensorwithcofibrant}
has a partial converse.

\begin{prop}
  \label{prop:cofibrantsource}
  Suppose that $j\co B \to B'$ is a map such that $(-) \otimes B$
  preserves acyclic cofibrations and that $(-) \otimes B'$
  preserves weak equivalences between cofibrant objects. If $i$ is an
  acyclic cofibration with cofibrant source, then the pushout-product
  map $i \boxtimes j$ is an equivalence.
\end{prop}

\begin{proof}
  Without loss of generality, let $i\co A \to A'$ be an acyclic
  cofibration and $j\co B \to B'$ a cofibration, with all four objects
  cofibrant. Then the pushout-product $i \boxtimes j$ is part of
  the following diagram:
  \[
    \xymatrix{
      & A' \otimes B \ar[d] \ar[dr] \\
      A \otimes B \ar[ur]^\sim \ar[dr] &
      P \ar[r]^-{i \boxtimes j} & A' \otimes B' \\
      &A \otimes B' \ar[u]_\sim \ar[ur]_\sim }
  \]
  The upper-left and lower-right maps are equivalences because they
  are obtained by tensoring an acyclic cofibration with the cofibrant
  objects $B$ and $B'$. The map $A \otimes B' \to P$ is the pushout of
  an acyclic cofibration, and so it is an acyclic
  cofibration. Therefore, by the 2-out-of-3 axiom the map $i \boxtimes
  j$ is an equivalence.
\end{proof}

The adunction isomorphism $\Hom_{\mathcal{M}}(X \otimes Y, Z) \cong
\Hom_{\mathcal{M}}(X, F^R(Y,Z))$, and similarly for the left, allows
us to rephrase the pushout-product axiom in multiple ways.
\begin{prop}[{\cite[4.2.2]{hovey-modelcategories}}]
  The following are equivalent for a model category $\mathcal{M}$ with
  a closed monoidal structure.
  \begin{enumerate}
  \item The model category $\mathcal{M}$ satisfies the pushout-product
    axiom.
  \item For a cofibration $i\co A \to B$ and a fibration $p\co X \to
    Y$ in $\mathcal{M}$, the induced map
    \[
      F^R(B,X) \to F^R(B,Y) \times_{F^R(A,Y)} F^R(A,X)
    \]
    is a fibration, which is acyclic if either $i$ or $p$ are.
  \item For a cofibration $i\co A \to B$ and a fibration $p\co X \to
    Y$ in $\mathcal{M}$, the induced map
    \[
      F^L(B,X) \to F^L(B,Y) \times_{F^L(A,Y)} F^L(A,X)
    \]
    is a fibration, which is acyclic if either $i$ or $p$ are.
  \end{enumerate}
\end{prop}

\begin{cor}[{\cite[4.2.5]{hovey-modelcategories}}]
  Suppose that $\mathcal{M}$ is a cofibrantly generated model
  category with a closed monoidal structure, a set $I$ of generating
  cofibrations and $J$ of generating acyclic cofibrations. Then the
  pushout-product axiom for $\mathcal{M}$ holds if and only if the
  pushout-product takes $I \times I$ to cofibrations in $\mathcal{M}$
  and takes both $I \times J$ and $J \times I$ to acyclic
  cofibrations.
\end{cor}

Because left Bousfield localization doesn't change the cofibrations in
a model structure, one is reduced to a few key verifications.

\begin{prop}
  Suppose that $\mathcal{M}$ is a (symmetric) monoidal closed model
  category with left Bousfield localization $\mathcal{M}'$. Then
  $\mathcal{M}'$ is compatibly a (symmetric) monoidal model category if and
  only if, for cofibrations $i$ and $j$ such that one is acyclic, the
  pushout-product map $i \boxtimes j$ is acyclic.

  If $\mathcal{M}'$ is cofibrantly generated, then it suffices to
  check that the pushout-product of a generating acyclic cofibration
  with a generating cofibration, in either order, is a weak
  equivalence.
\end{prop}

\begin{rmk}
  If the generating cofibrations and generating acyclic cofibrations
  of $\mathcal{M}'$ have cofibrant source, then by
  Proposition~\ref{prop:cofibrantsource} we only need to show that
  tensoring with the sources or target of any map in $I$ or $J$
  takes generating cofibrations in $\mathcal{M}'$ to weak
  equivalences.
\end{rmk}

\begin{rmk}
  Bousfield localization of stable model categories has been more
  extensively studied by Barnes and Roitzheim
  \cite{barnes-roitzheim-stablelocalization,
    barnes-roitzheim-homologicallocalization}. To have homotopical
  control over \emph{commutative} algebra objects in a symmetric
  monoidal model category, one needs to obtain control over the
  extended power constructions; see \cite{white-commutativemonoids}.
\end{rmk}

\subsection{Monoidal $\infty$-categories}

We will begin by giving a brief background on monoidal structures on
$\infty$-categories which is light on technical details. 

Recall that a \emph{multicategory}
$\mathcal{O}$ is equivalent to the following data:
\begin{enumerate}
\item a collection of objects of $\mathcal{O}$;
\item for any object $Y$ and indexed set of objects
  $\{X_s\}_{s \in S}$ of $\mathcal{O}$, a space
  $\Map_{\mathcal{O}}(\{X_s\}_{s \in S};Y)$ of multimaps; and
\item for a surjection $p\co S \to T$ of finite sets, natural
  composition maps
  \[
    \Map_{\mathcal{O}}(\{Y_t\}_{t \in T};Z) \times \prod_{t \in T}
    \Map_{\mathcal{O}}(\{X_s\}_{s \in p^{-1}(t)}; Y_t) \to
    \Map_{\mathcal{O}}(\{X_s\}_{s \in S}; Z)
  \]
  that are compatible with composing surjections $S \to T \to U$.
\end{enumerate}
  
\begin{rmk}
  As a special case, for $\sigma$ a permutation of $S$ there is an
  isomorphism
  $\Map_{\mathcal{O}}(\{X_s\}_{s \in S};Y) \to
  \Map_{\mathcal{O}}(\{X_{\sigma(s)}\}_{s \in S};Y)$, and the
  composition operations are appropriately equivariant with respect to
  these isomorphisms.
\end{rmk}

For such a multicategory, we could give a prototype definition of an
$\mathcal{O}$-monoidal $\infty$-category $\mathcal{C}$ as an enriched
functor from $\mathcal{O}$ to $\infty$-categories. This data
specifies, for each object $X$ of $\mathcal{O}$, a category
$\mathcal{C}_X$. For each object $Y$ and indexed set
$\{X_s\}_{s \in S}$ of objects, there is a specified continuous map
from $\Map_{\mathcal{O}}(\{X_s\}_{s \in S};Y)$ to the space of functors
$\prod_{s \in S} \mathcal{C}_{X_s} \to \mathcal{C}_Y$. Moreover, these
maps must be compatible with composition on both sides.

The definition of an $\infty$-operad $\mathcal{O}$ and an
$\mathcal{O}$-monoidal $\infty$-category $\mathcal{C}$ is slightly
different from this \cite[\S 2.1]{lurie-higheralgebra}. Roughly, it is
an \emph{unstraightened} definition where the spaces of multimaps in
$\mathcal{O}$ and the product functors on $\mathcal{C}$ are only
specified up to a contractible space of choices; the technical details
are related in spirit to Segal's work
\cite{segal-categoriescohomology}. Even though the functors induced
from $\mathcal{O}$ are specified only up to contractible
indeterminacy, it still makes sense to ask about compatibility of the
monoidal structure with localization.

The following result very general result encodes the situations under
which homotopical localization is compatible with monoidal structures.
\begin{thm}[{\cite[2.2.1.9]{lurie-higheralgebra}}]
  Let $\mathcal{O}^\otimes$ be an $\infty$-operad and let
  $\mathcal{C}$ be an $\mathcal{O}$-monoidal
  $\infty$-category. Suppose that for all objects $X$ of
  $\mathcal{O}$ we have a localization functor
  $L_X\co \mathcal{C}_X \to \mathcal{C}_X$, and that for any map
  $\alpha\co \{X_s\}_{s \in S} \to Y$ in $\mathcal{O}^\otimes$ the induced
  functor $\prod_{s \in S} \mathcal{C}_{X_s} \to \mathcal{C}_Y$ preserves
  $L$-equivalences in each variable. Then there exists an
  $\mathcal{O}$-monoidal structure on the category $L\mathcal{C}$ of
  local objects making the localization
  $L\co \mathcal{C} \to L\mathcal{C}$ into an $\mathcal{O}$-monoidal
  functor.
\end{thm}

\begin{cor}
  Suppose that $\mathcal{C}$ is a (symmetric) monoidal
  $\infty$-category and that $L$ is a localization functor on
  $\mathcal{C}$ such that $L(X \otimes Y) \to L(LX \otimes LY)$ is
  always an equivalence. Then the subcategory $L\mathcal{C}$ of local
  objects has the structure of a (symmetric) monoidal
  $\infty$-category and any localization functor $L$ has the structure
  of a (symmetric) monoidal functor.
\end{cor}

\begin{exam}
  In the category of spaces, we can use the mapping space adjunctions
  and find that for any $S$-local object $Z$, we have
  \begin{align*}
    \Map(X \times Y, Z) &\simeq \Map(X,\Map(Y,Z))\\
                        &\simeq \Map(X,\Map(LY,Z)\\
                        &\simeq \Map(X \times LY, Z)
  \end{align*}
  and similarly on the other side, showing that $LX \times LY$ is a
  localization of $X \times Y$. This gives the cartesian product on
  spaces the special property that it is compatible with \emph{all}
  localization functors.
\end{exam}

\begin{exam}
  Fix an $E_n$-operad $\mathcal{O}$ and an $\mathcal{O}$-algebra $B$
  in spaces representing an $n$-fold loop space. Consider the category
  $\mathcal{C}$ of functors $B \to \mathcal{S}$, viewed as local
  systems of spaces over $B$. Then the category $\mathcal{C}$ has a
  \emph{Day convolution}, developed by Glasman \cite{glasman-day} in
  the $E_\infty$-case and by Lurie \cite[{\S
    2.2.6}]{lurie-higheralgebra} in general, making $\mathcal{C}$ into
  an $\mathcal{O}$-monoidal category. The category $\mathcal{C}$ is
  equivalent (via unstraightening) to the category of spaces over $B$.
  In these terms the $\mathcal{O}$-monoidal structure is given by maps
  \begin{align*}
    \mathcal{O}(n) &\to \Map(B^n, B)\\
                   &\to \Fun((\mathcal{S}/B)^n,\mathcal{S}/B)
  \end{align*}
  that respect composition. Here $f \in \mathcal{O}(n)$ first goes to
  $f\co B^n \to B$, then to the functor sending $\{X_i \to B\}$ to the
  map $\prod X_i \to B^n \too{f} B$. An $\mathcal{O}$-algebra in
  $\mathcal{C}$ is equivalent to an $E_n$-space $X$ with a map $X \to
  B$ of $E_n$-spaces.

  Suppose $L$ is a Bousfield localization on spaces, and consider the
  associated pointwise localization on the functor category
  $\mathcal{C}$ (which corresponds to the fiberwise localization on
  spaces over $B$). All operations in $\mathcal{O}$ are, up to
  homotopy, composites of the binary multiplication operation, and so
  it suffices to show that this preserves localization. However, if
  the maps $X_i \to B$ have homotopy fibers $F_i$, then the homotopy
  fiber of the map $X_1 \times X_2 \to B \times B \to B$ is, up to
  equivalence, the geometric realization of the bar construction
  \[
    B(F_1, \Omega B, F_2).
  \]
  Since any localization preserves homotopy colimits and products of
  spaces, this bar construction preserves it also. Therefore,
  fiberwise localization is an $E_n$-monoidal functor on the category
  of spaces over $B$.\footnote{For \emph{grouplike} $E_n$-spaces over
    a grouplike $B$, this is roughly the statement that we can take
    $n$-fold classifying spaces, apply the fiberwise localization, and
    then take $n$-fold loop spaces.}
\end{exam}

\bibliography{../masterbib}
\end{document}